\documentclass[12pt]{article}

\usepackage{url}
\usepackage{graphicx}
\usepackage{hyperref}
\usepackage{multirow}
\usepackage{amsmath}
\usepackage{amsthm}
\usepackage{amssymb}
\usepackage{amsfonts}
\usepackage{stmaryrd}

\hypersetup{
	colorlinks,
	citecolor=black,
	filecolor=black,
	linkcolor=black,
	urlcolor=black
}

\newcommand{\quand}{\quad \text{ and } \quad}

\newcommand{\ii}{{\rm{i}}}

\newcommand{\tmu}{{\mu_h}}

\newcommand{\Pio}{\Pi_{0}}
\newcommand{\ur}{{u_{h,\star}^{n+1}}}

\newcommand{\vertiii}[1]{{\left\vert\kern-0.25ex\left\vert\kern-0.25ex\left\vert #1 
		\right\vert\kern-0.25ex\right\vert\kern-0.25ex\right\vert}}

\makeatletter
\DeclareRobustCommand\widecheck[1]{{\mathpalette\@widecheck{#1}}}
\def\@widecheck#1#2{%
	\setbox\z@\hbox{\m@th$#1#2$}%
	\setbox\tw@\hbox{\m@th$#1%
		\widehat{%
			\vrule\@width\z@\@height\ht\z@
			\vrule\@height\z@\@width\wd\z@}$}%
	\dp\tw@-\ht\z@
	\@tempdima\ht\z@ \advance\@tempdima2\ht\tw@ \divide\@tempdima\thr@@
	\setbox\tw@\hbox{%
		\raise\@tempdima\hbox{\scalebox{1}[-1]{\lower\@tempdima\box
				\tw@}}}%
	{\ooalign{\box\tw@ \cr \box\z@}}}
\makeatother

\usepackage{mathtools}
\mathtoolsset{showonlyrefs}

\newcommand{\hf}{\frac{1}{2}}

\newcommand{\veps}{\varepsilon}

\newcommand{\nm}[1]{\| #1 \|}

\newcommand{\ip}[1]{\langle #1 \rangle}

\newtheorem{THM}{Theorem}[section] 

\newtheorem{ASSP}{Assumption}[section]
\newtheorem{PROP}{Proposition}[section]

\newtheorem{LEM}{Lemma}[section]
 \newtheorem{COR}{Corollary}[section]
\newtheorem{REM}{Remark}[section]
\newtheorem{examp}[subsubsection]{Example}

\begin{document}
\baselineskip=2pc
\begin{center}
	\Large {\bf Error analysis of Runge--Kutta discontinuous Galerkin methods for linear 
time-dependent partial differential equations\footnote{Acknowledgment: The authors would like to 
express their appreciation to Prof. Qiang Zhang at Nanjing University, whose insightful 
comments helped improve this paper. ZS also wants to thank Prof. Yulong Xing and Prof. Ruchi Guo 
at The Ohio State University for helpful discussions. }}
\end{center}
\centerline{
	Zheng Sun\footnote{Department of Mathematics, The Ohio State University,
		Columbus, OH 43210, USA. E-mail: sun.2516@osu.edu.} and 
	Chi-Wang Shu\footnote{Division of Applied Mathematics, Brown University,
		Providence, RI 02912, USA. E-mail: chi-wang\_shu@brown.edu.
        Research supported by NSF grant DMS-1719410.}
}
\vspace{.2in}
\centerline{\bf Abstract} \medskip

In this paper, we present error estimates of fully discrete Runge--Kutta discontinuous Galerkin (DG) schemes for linear time-dependent partial differential equations. The analysis applies to explicit Runge--Kutta time discretizations of any order. For spatial discretization, a general discrete operator is considered, which covers various DG methods, such as the upwind-biased DG method, the central DG method, the local DG method and the ultra-weak DG method. We obtain error estimates for stable and consistent fully discrete schemes, if the solution is sufficiently smooth and a spatial operator with certain properties exists. Applications to schemes for hyperbolic conservation laws, the heat equation, the dispersive equation and the wave equation are discussed. In particular, we provide an alternative proof of optimal error estimates of local DG methods for equations with high order derivatives in one dimension, which does not rely on energy inequalities of auxiliary unknowns.\bigskip

 \bigskip


\section{Introduction}
\setcounter{equation}{0}
Let $\Omega\subset \mathbb{R}^d$ be the spatial domain. $u=u(x,t):\Omega\times(0,+\infty)\to \mathbb{R}^m$ is a vector-valued function and $L  =\sum_{|\alpha|\leq q} a_\alpha(x) D^\alpha: \Omega \to \mathbb{R}^m$ is a $q$th order differential operator.
The time-dependent partial differential equation (PDE)
\begin{equation}\label{eq-cont}
\partial_t u = Lu
\end{equation}
 is usually discretized in a two-step procedure. The first step is to apply spatial discretization to obtain a method-of-lines scheme
\begin{equation}\label{eq-semi}
\partial_tu_h= L_hu_h.
\end{equation}	
The resulted linear autonomous system is then discretized with a time integrator in the second step. In this paper, we are particularly interested in the case that $L_h$ arises from discontinuous Galerkin (DG) finite element approximations, although
the analysis also applies to other spatial discretization methods. For time discretizations, we consider explicit Runge--Kutta (RK) time stepping methods, which are in the form of a truncated Taylor series when applied to \eqref{eq-semi}. The fully discrete scheme can be written as
\begin{equation}\label{eq-fully}
u_h^{n+1} = R_s(\tau L_h ) u_h^n,\qquad R_s(\tau L_h) = \sum_{i = 0}^s\alpha_i (\tau L_h)^i.
\end{equation}
Here $s$ is the number of stages, $\tau$ is the time step size and $\{\alpha_i\}_{i=0}^s$ are constants dependent on the choice of the RK method. We will perform error estimates of the fully discrete scheme \eqref{eq-fully} under certain assumptions, and provide examples to various DG schemes for hyperbolic conservation laws, the heat equation, the dispersive equation and the wave equation, etc. 

There has been a long history on analyzing convergence properties of the fully discrete schemes for linear time-dependent PDEs. The equivalence theorem given by Lax and Richtmyer in 1956 states that a consistent finite difference approximation of a linear equation converges if and only if it is stable \cite[Section 8]{lax1956survey}. Then with a recurrent argument, a unified error estimate based on local truncation error analysis can be established for general linear finite difference schemes \cite[Theorem 4.2.3]{gustafsson2013time}. However, the same procedure can not be applied to Galerkin schemes due to the phenomenon of supraconvergence, in that the finite difference schemes reformulated from the Galerkin schemes may exhibit lower order accuracy or even be inconsistent when measured with truncation error \cite{kreiss1986supra,zhang2003analysis}. Instead, arguments with an appropriately constructed spatial projection (or interpolation) operator are usually
used, replacing the local truncation error analysis in space. For parabolic equations and second order hyperbolic equations, the steady state problems correspond to an elliptic equation, and the elliptic projection can be used to derive error estimates. Along this stream of research, error estimates have been obtained for Galerkin schemes with multistep \cite{gekeler1976linear,dougalis1979multistep,baker1980multistep} and (implicit) RK time discretizations \cite{keeling1990galerkin}. 

The DG methods are a class of finite element methods using discontinuous piecewise polynomial spaces. It was first proposed by Reed and Hill in \cite{reed1973triangular} for solving the transport equation and then received its major development in a series of work by Cockburn et al. for solving hyperbolic conservation laws \cite{rkdg1,rkdg2,rkdg3,rkdg4,rkdg5}. After that, based on successful numerical experiments by Bassi and Rebay \cite{bassi1997high}, Cockburn and Shu proposed the local DG (LDG) method for solving convection-diffusion systems \cite{cockburn1998local}, which was soon generalized for equations with higher order derivatives \cite{xu2010local}. In the past decades, different variants of
DG methods have been developed, such as the central DG method \cite{liu2007central}, the direct DG method \cite{liu2009direct} and the ultra-weak DG method \cite{cheng2008discontinuous}, just to name a few. 
Error estimates of these DG methods have been studied in various of contexts, including hyperbolic conservation laws \cite{zhang2004error,zhang2006error,zhang2010stability,meng2016optimal,sun2017stability,liu2018optimal,liu2020optimal}, convection-diffusion systems \cite{cockburn1998local,wang2015stability,liu2015optimal,wang2016stability,cheng2017application}, the KdV equation \cite{yan2002local,xu2007error,bona2013conservative},  the Camassa--Holm equation \cite{xu2008local}, the wave equation  \cite{xing2013energy,chou2014optimal},  the improved Boussinesq equation \cite{li2020energy}, high odd order equations \cite{xu2012optimal} and high even order equations \cite{dong2009analysis}, etc. 

The method-of-lines DG schemes are usually discretized with an explicit RK time integrator and the resulted fully discrete schemes are referred to as Runge--Kutta discontinuous Galerkin (RKDG) schemes. Besides the simplicity of implementation, the popularity of explicit RK methods is also due to its compatibility with limiters to preserve certain properties of continuum equations and to achieve better robustness. One of the difficulties on error analysis of RKDG schemes beyond method of lines is to establish the fully discrete $L^2$ stability with explicit RK time discretizations. Although this is well understood for diffusive problems for general explicit RK schemes \cite{gottlieb2001strong}, the stability for nearly energy-conserving systems is nontrivial and sometimes a stricter time step constraint has to be enforced. Recently, based on a few earlier work \cite{tadmor2002semidiscrete,zhang2010stability,sun2017rk4,ranocha2018l_2}, a systematic stability analysis has been performed by Sun and Shu in \cite{sun2019strong} for general linear semi-negative operators and also by Xu et al. in the context of RKDG schemes for linear conservation laws \cite{xu2019L2,xu2019superconvergence}. A few stabilization approaches have also been proposed recently \cite{ketcheson2019relaxation,sun2019enforcing}. Thanks to these results, the involved energy estimation in the error analysis can be avoided by referring to stability properties as a black box. 

The other issue is to find suitable projection operators for error analysis. For most cases, the projections constructed for semidiscrete DG schemes can be directly used in the fully discrete context. While for LDG methods, the projections are usually defined for all auxiliary unknowns in the mixed formulation and can not be applied to the current framework. Motivated by the construction of the elliptic projection, we define the operator by formally solving the steady state problem. The technicality is that the kernel of $L_h$ can be nonzero and the inverse has to be defined on a suitable
subspace. The resulted operator (detailed in Section \ref{sec-xu}) works directly with the primal formulation in one dimension, and it indeed retrieves the initial projection used in \cite{xu2012optimal} for the third order dispersive equation. As a result, optimal error estimates can be obtained without energy inequalities of the 
auxiliary unknowns, which simplifies the proof in \cite{xu2012optimal} for odd order equations and provides an alternative interpretation of the proof in \cite{dong2009analysis} for even order equations in one dimension.

This paper is built upon above ingredients. We show that for sufficiently smooth exact solutions, if there exists a spatial operator with certain properties, a stable and consistent fully discrete RKDG scheme has the convergence rate $\mathcal{O}(\tau^p + h^{k+k'})$. Here $\tau$ is the time step size, $h$ is the spatial mesh size, $p$ is the linear accuracy order of the time integrator\footnote{We refer this as the linear order throughout the paper.}, $k$ is the polynomial degree and $k'\in[0,1]$ depends on particular problems. It 
is worth mentioning that the required regularity is independent of the
number of stages of the RK method, which is achieved by using a carefully chosen reference solution \eqref{eq-ur} in the proof. Applications to various DG schemes are given in the paper. We also provide examples with continuous Galerkin (CG) finite element methods and with Fourier Galerkin (FG) methods for possible extensions to other types of spatial discretizations. Finally, to compare our work with error analysis of ordinary differential equations, we discuss a different approach, in which we assume the error of the method-of-lines scheme and compare the fully discrete solution with the semidiscrete solution for error estimates. This argument requires the construction of a different projection operator and 
currently it applies only to a few schemes. When writing this paper, the authors are inspired by the work of Xu et al. on error analysis of the fourth order RKDG scheme for linear hyperbolic conservation laws \cite{xu2019error}. In their recent preprint \cite{xu2019superconvergence}, some techniques and results have been further explored for superconvergence analysis. Compared with \cite{xu2019error,xu2019superconvergence}, our work includes a larger class of DG methods and also applies to problems beyond hyperbolic conservation laws. The language in this paper also shares similarity with that by Chen in \cite{chen2009unified}, in which the author explained the Lax equivalence theorem under a general framework and provided examples of different schemes for steady state elliptic equations. Compared with \cite{chen2009unified}, our paper emphasizes more on analysis of time-dependent problems by going through the recurrence relationship between time steps.

The rest of the paper is organized as follows. We start with clarifying notations, preliminaries and assumptions in Section \ref{sec-not}. Then error estimates of the semidiscrete scheme and the fully discrete scheme are given in Section \ref{sec-errest}. After that, we apply the fully discrete error analysis to various DG schemes, as well as some CG and FG schemes, in Section \ref{sec-app}. The error analysis built directly upon semidiscrete results is discussed in Section \ref{sec-ano}. Finally, we close the paper with conclusions in Section \ref{sec-con}. 


\section{Notations and assumptions}\label{sec-not}
\setcounter{equation}{0}
Let $V = L^2(\Omega; \mathbb{R}^m)$ be the space of interest, equipped with the inner product  $\ip{\cdot,\cdot}$ and the norm $\nm{\cdot} = \sqrt{\ip{\cdot,\cdot}}$. $V_h \subset V$ is the space of discrete solutions. To be more specific, we have $u(\cdot, t) \in V$ and $u_h(\cdot, t), u_h^n(\cdot) \in V_h$. As a convention, we will omit the variable $x$ and denote by $u(t) = u(x,t)$ and $u_h(t) = u_h(x,t)$ when there is no confusion. Throughout the paper, we use $\Pio:V\to V_h$ to represent the $L^2$ projection. $\Pi$ is a projection or an interpolation operator, which maps a sufficiently smooth function to $V_h$.  $L_h:V_h\to V_h$ is the discrete operator approximating $L$. For simplicity, uniform time steps are assumed and $t^n = n \tau $. We also assume $\tau \leq 1$ and $\tau \nm{L_h} \leq \lambda <1$. 

\begin{PROP}[Gr\"onwall's inequalities]\label{prop-gronwall}
	Let $a$ be a nonnegative constant and 
	\begin{equation}
	 \sigma(a,t) = \left\{\begin{matrix}\frac{e^{at}-1}{a}, & a>0,\\
	 t, &a = 0.\end{matrix}\right.
	 \end{equation}
	\begin{enumerate}
	\item Suppose $\frac{d}{dt} y(t) \leq a y(t) + b(t)$. Then we have
	\begin{equation}\label{eq-gronwall}
		y(t) \leq e^{at} y(0) + \int_{0}^t e^{a(t-r)}b(r)dr \leq e^{at} y(0) + \sigma(a,t)\sup_t |b(t)|.
	\end{equation}
	\item Suppose $	y_{n+1} \leq a y_{n} + b_{n}$. Then with the convention $0^0 = 1$, we have 
	\begin{equation}\label{eq-dgronwall}
	y_{n+1} \leq a^{n+1} y_0 + \sum_{i=0}^{n} a^{i} b_{n-i}\leq a^{n+1} y_0 +\left( \sum_{i=0}^{n} a^{i} \right)\sup_n |b_{n}|.
	\end{equation}
	\end{enumerate}
\end{PROP}
\begin{REM}\label{rmk-dg}
	We will apply Gr\"onwall's inequalities with $a = \nm{R_s(\tau L_h)}$ under the assumption $\nm{R(\tau L_h)} \leq 1 + \tmu  \tau$ in the error analysis. By using the fact $(1+y)^{1/y} \leq e$, $\forall y>0$, it can be shown that $\nm{R_s(\tau L_h)}^{n+1}  \leq e^{\tmu t^{n+1}}$ and $
	\sum_{i=0}^n\nm{R_s(\tau L_h)}^i \leq  \sigma(\tmu, t^{n+1}) \tau^{-1}$.
\end{REM}

Since the error estimates rely on the Lax--Wendroff procedure, in that we convert the temporal operator $\partial_t$ into the spatial operator $L$, we need to assume the exact solution has sufficient regularity to justify this conversion. 
\begin{ASSP}[Regularity of $u$]\label{assp-rg}
	$u$ is sufficiently smooth, such that 
	$\{\partial_t^i u\}_{i=1}^{p+1}$, $\{L^i u\}_{i=1}^{p+1}$ and $\{\Pi\partial_t^{i} u(\cdot,t)\}_{i=0}^{p+1}$ are well-defined and bounded in $L^\infty$ norm. Moreover,  $\partial_t^i u = L^i u$, $\forall 1\leq i \leq p+1$.
\end{ASSP}
\begin{REM}
Note the regularity assumption is independent of the stage number $s$. In many cases, it can be satisfied with $u \in C^{q(p+1)}(\Omega\times (0,+\infty);\mathbb{R}^m)$.
\end{REM}

Stability of the scheme, which bounds the error growth in time, plays a crucial role in the error analysis. Here we assume the semiboundedness of $L_h$ in Assumption \ref{assp-Lh}, and stability of the semidiscrete scheme follows as a consequence. Stability of the fully discrete scheme is also based on Assumption \ref{assp-Lh}, with additional time step constraints. For clarity, we separately state the assumption on fully discrete stability in Assumption \ref{assp-RK} and comment on its connection with Assumption \ref{assp-Lh} in Remark \ref{rem-timestep}.  

\begin{ASSP}[Semiboundedness of $L_h$]\label{assp-Lh}
	There exists a constant $\mu \geq 0$, such that 
	\begin{equation}
	\ip{L_h v_h,v_h} \leq {\mu} \nm{v_h}^2, \qquad \forall v_h \in V_h.
	\end{equation}
\end{ASSP}
\begin{PROP}[Semidiscrete stability]\label{prop-edecay}
	Under Assumption \ref{assp-Lh}, we have
	$\nm{u_h(t)} \leq e^{\mu t} \nm{u_h(0)}$.
\end{PROP}
\begin{ASSP}[Fully discrete stability]\label{assp-RK}
		There exists a constant $\tmu \geq 0$, such that under the time step constraint $\tau \nm{L_h} \leq \lambda$, we have
		$\nm{R_s(\tau L_h)} \leq 1+ \tmu \tau$.
		Here $\lambda$ can either be a constant or depend on the mesh size $h$.
\end{ASSP}

\begin{REM}[On the time step constraint]\label{rem-timestep}
	Using inverse estimates, we can usually show that $\nm{L_h} \lesssim h^{-q}$ for a $q$th order differential operator. When $\lambda$ is constant, the time step constraint is in the form of $\tau \lesssim \lambda h^q$, which is the practically assumed time step size. 
	
	Assume $\ip{L_h v_h,v_h} \leq \mu \nm{v_h}^2$ with $\mu\leq 0$. For a diffusive problem that leads to a coercive $L_h$, namely $\ip{L_h v_h,v_h}\leq -\eta \nm{L_h v_h}^2$ with $\eta>0$, 
then all $p$-stage $p$th order explicit RK methods are stable under the constraint $\tau \lesssim \eta$  \cite{gottlieb2001strong}. For general cases, $p$-stage $p$th order RK methods with $p = 3,7,11,\cdots$ \cite{sun2019strong}, and RK methods combining two steps of $p$-stage $p$th order RK methods, with $p = 4,8,12,\cdots$ \cite{sun2017rk4,xu2019L2,xu2019superconvergence}, 
	are stable with constant $\lambda$. Analysis in \cite{xu2019L2,xu2019superconvergence} also implies that
	all $p$-stage $p$th order explicit RK methods are stable under the time step constraint $\tau \nm{L_h}^2 \lesssim 1$, which is effectively $\lambda \lesssim h^{q}$. We refer to their original papers for a less restrictive time step estimation. One can also expect similar results for $\mu >0$. See, for example, \cite{sun2017stability}.
\end{REM}

Another ingredient for error estimates is the consistency of the scheme. Consistency of the RK time discretization can be examined with local truncation error analysis. See Assumption \ref{assp-rkacc}. The consistency of the spatial operator is defined based on the existence of a projection or interpolation operator, which is detailed in Assumption \ref{assp-hard}. 

\begin{ASSP}[Consistency of the RK method]\label{assp-rkacc}
	The RK method is $p$th order accurate, $p\geq 1$. More specifically, we assume 
	$\alpha_i = \frac{1}{i!}$, $\forall i\leq p$, and   $\alpha_{p+1} \neq \frac{1}{(p+1)!}$ in \eqref{eq-fully}.  
\end{ASSP}

\begin{ASSP}\label{assp-hard}
	There exists a linear operator $\Pi$ such that $\mathrm{Range}(\Pi) \subset V_h$ and 
	\begin{enumerate}
		\item (Approximation property of $\Pi$). 
		$
		\nm{\left(I  - \Pi\right) L^i w(t)} \leq E_{\rm{T}}(i, t,h)
		$.
		\item (Consistency of $L_h$).
		$\left|\ip{\left( L - L_h \Pi\right) L^{i-1} w(t), v_h} \right|\leq E_{\rm{S}}(i, t,h) \nm{v_h}$, $\forall v_h \in V_h$.
	\end{enumerate}
	We denote by
	$
	E_i(h) = \left\{
	\begin{matrix}
	\sup_{t} \left(E_{\rm{S}}(i,t,h) + E_{\rm{T}}(i,t,h)\right), & i \geq 1,\\
	\sup_{t} E_{\rm{T}}(i,t,h), & i =0.\\
	\end{matrix}
	\right.
	$
\end{ASSP}
\begin{REM}
	We usually refer to $\Pi$ as a projection for consistency with existing literature in the DG community. However, in this paper, we do not assume $\Pi$ is an actual projection in the mathematical sense, namely $\Pi^2 = \Pi$. In fact, $\Pi$ can be undefined on $V_h$.
\end{REM}
\begin{REM}
	As will be detailed in Theorem \ref{thm-semi}, one only needs to consider $i = 1$ in the error estimates of semidiscrete schemes. In this case, Assumption \ref{assp-hard} can be rephrased as
	\begin{align}
	\nm{\left(I  - \Pi\right) \partial_t u(t)} \leq&\; E_{\rm{T}}(1, t,h),\label{eq-hard-E-i1}\\
	\left|\ip{\left( L - L_h \Pi\right) u(t), v_h} \right|\leq &\;E_{\rm{S}}(1, t,h) \nm{v_h}, \qquad \forall v_h \in V_h.\label{eq-hard-projprop-i1}
	\end{align}	
	 \eqref{eq-hard-E-i1} is the standard approximation result of $\Pi$. The approximation of $Lu$ is concerned in \eqref{eq-hard-projprop-i1}. When $E_{\rm{S}}$ is of the same order as $E_{\rm{T}}$, \eqref{eq-hard-E-i1} is referred to as a superconvergence property. See, for example, \cite{dong2009analysis}. 
\end{REM}

\section{Error estimates}\label{sec-errest}

\subsection{Semidiscrete scheme}\label{sec-semi}

\begin{THM}[Error estimates of the semidiscrete scheme] \label{thm-semi}
	Under Assumption \ref{assp-rg} with $p = 0$, Assumption \ref{assp-Lh} and Assumption \ref{assp-hard}, the semidiscrete scheme \eqref{eq-semi} satisfies the error estimate
	\begin{equation}
			\nm{u(t) - u_h(t)} \leq  e^{\mu t} \nm{\Pi u(0)-u_h(0)} + (\sigma(\mu, t)+1)E_1(h) .
	\end{equation}
\end{THM}

\begin{proof}
	Subtracting \eqref{eq-semi} from \eqref{eq-cont} gives
	\begin{equation}
	\partial_t (u-u_h) = L u- L_h u_h.
	\end{equation}
	Then by adding and subtracting terms, we have
	\begin{equation}\label{eq-semisplit} 
	\partial_t \xi  =  L_h \xi + \veps   - \partial_t \eta, 
	\end{equation}
	where
	$\xi  = \Pi u - u_h$, $\veps  = \left(L - L_h \Pi \right) u$ and $\eta  = \left(I - \Pi\right) u$.
	Take inner product of \eqref{eq-semisplit} with $\xi$, and it yields that
	\begin{equation}
	\begin{aligned}
	\hf \frac{d}{dt}\nm{\xi}^2 
	= \ip{L _h\xi , \xi }+ \ip{ \veps ,\xi} - \ip{ \partial_t \eta,\xi}.
	\end{aligned}
	\end{equation}
	Note we have $\ip{L _h\xi , \xi }\leq \mu \nm{\xi}^2$ from Assumption \ref{assp-Lh}, 
	$\ip{\veps , \xi } \leq E_{\rm{S}}(1,t,h)\nm{\xi}$ from Assumption \ref{assp-hard}, and  
	\begin{equation}
	\ip{\partial_t \eta, \xi} = \ip{ (I-\Pi) \partial_tu, \xi }= \ip{ (I-\Pi) Lu, \xi }\leq E_{\rm{T}}(1,t,h)\nm{\xi}
	\end{equation}
	from Assumption \ref{assp-rg} and Assumption \ref{assp-hard}. Therefore, 
	\begin{equation}
	\hf \frac{d}{dt}\nm{\xi}^2 \leq  \left(\mu \nm{\xi} + E_1(h) \right) \nm{\xi},
	\end{equation}
	which implies 
	$\frac{d}{dt}\nm{\xi} \leq \mu \nm{\xi} + E_1(h)$.
	One can then use Gr\"onwall's inequality \eqref{eq-gronwall} to obtain
	\begin{equation}\label{eq-semi-xi}
	\nm{\xi(t)}	\leq e^{\mu t} \nm{\xi(0)} + \sigma(\mu, t) E_1(h). 
	\end{equation}
	Finally, after applying triangle inequality, we have
	\begin{equation}
	\nm{u(t) - u_h(t)} \leq  \nm{\eta(t)} + \nm{\xi(t)} \leq e^{\mu t} \nm{\xi(0)} + \left(\sigma(\mu,t) +1\right)E_1(h). 
	\end{equation}
\end{proof}

\subsection{Fully discrete scheme}\label{sec-fully}

\subsubsection{Main results}
 
\begin{THM}[Error estimates of the fully discrete scheme] \label{thm-main-hard}
	Under Assumption \ref{assp-rg},  Assumption \ref{assp-RK}, Assumption \ref{assp-rkacc} and Assumption \ref{assp-hard}, the fully discrete scheme \eqref{eq-fully} satisfies the error estimate 
	\begin{equation}
	\nm{u(t^n) - u_h^n}\leq e^{\tmu t^n}  \nm{ \Pi u(0) - u_h^0  }+E_0(h)+ C_\star\sigma(\tmu, t^{n})\left(\sum_{i =1}^{p+1} \tau^{i-1}E_i(h) +  \sup_{t} \nm{\partial_t^{p+1} u}\tau^p\right) .
	\end{equation}
	Here $\sigma(\cdot,\cdot)$ is defined in Proposition \ref{prop-gronwall} and $C_\star = 2\sum_{i=0}^{s}|\alpha_i|$. 
\end{THM}

As a consequence, we have the following error estimate when $u$ is sufficiently smooth.
\begin{COR}\label{cor-err}
	Let $k'\in[0,1]$ and $k$ be parameters associated with $V_h$ and $L_h$. Suppose the following assumptions hold. 
	\begin{enumerate}
		
		\item There exists $\lambda>0$ such that 
		$\nm{R(\tau L_h)}\leq 1+ \tmu \tau$ for all $\tau \leq \lambda h^q$.
		\item There exists a linear operator $\Pi$ such that for sufficiently smooth $w$, we have
		\begin{align}
		\nm{w - \Pi w}\leq &\; C_{\rm{T}}h^{k+1},\label{eq-CT}\\
		\ip{(L  - L_h\Pi) w, v_h }\leq&\; C_{\rm{S}} h^{k+k'}\nm{v_h},\label{eq-CS}
		\end{align}
		where $C_{\rm{S}}$ and $C_{\rm{T}}$ are constants dependent on the Sobolev norm of $w$. 
	\end{enumerate}
	Then when $u$ is sufficiently smooth, we have
	\begin{equation}\label{cor-main-hard}
	\nm{u(t^n) - u_h^n}\leq e^{\tmu t^n}  \nm{ \Pi u(0) - u_h^0  }+C_{\rm{E}}(\sigma(\tmu, t^{n})+1)\left(h^{k+k'} + \tau^p\right),
	\end{equation}
	where $C_{\rm{E}}$ is a constant dependent on $\{\alpha_i\}_{i=0}^s$ and the Sobolev norm of $u$. 
	
\end{COR}

\subsubsection{Proof of Theorem \ref{thm-main-hard}}

The main step for proving Theorem \ref{thm-main-hard} is to track the discrete error
\begin{equation}\label{eq-xi0}
	\xi_0^{n+1} = \Pi u (t^{n+1}) - u_h^{n+1}.
\end{equation}
To avoid extra regularity assumptions, we modify the reference $u^{n+1}_\star = \sum_{i=0}^{s}\alpha_i (\tau L)^i u(t^n)$ as
\begin{equation}\label{eq-ur}
\ur = \Pi \left(\sum_{i = 0}^p \alpha_i(\tau L)^i u(t^n)\right)+ \sum_{i = p+1}^s \alpha_i\tau^i L_h ^{i-(p+1)}\Pi L^{p+1} u(t^n).
\end{equation}
Then the discrete error $\xi_0^{n+1}$ admits the splitting 
\begin{equation}\label{eq-split}
	\xi_0^{n+1} =\left(\Pi u(t^{n+1}) - \ur\right) + \left(\ur - u_h^{n+1}\right):= \rho_h^{n+1} + \omega_h^{n+1}.
\end{equation}
By expanding $\ur$ with \eqref{eq-ur} and $u_h^{n+1}$ with \eqref{eq-fully}, we have 
\begin{equation}
\omega_h^{n+1} =  \sum_{i = 0}^p\alpha_i\tau^i (\Pi L^i u(t^n)- L_h^i u_h^n) + \sum_{i = p+1}^s\alpha_i \tau^i \left(L_h^{i-(p+1)}\Pi L^{p+1} u(t^n) - L_h^{i} u_h^n\right),
\end{equation}
which motivates us to define the following quantity
\begin{equation}\label{eq-xi}
\xi_i^n = 
\left\{
\begin{matrix}
\Pi L^i u(t^n)- L_h^i u_h^n, &0\leq i\leq {p+1},\\
L_h^{i-(p+1)}\Pi L^{p+1} u(t^n)- L_h^i u_h^n, &p+1\leq i\leq s.\\
\end{matrix}\right.
\end{equation}
Note the notation is consistent with that of $\xi_0^{n+1}$ in \eqref{eq-xi0}. Then \eqref{eq-split} can be written as
\begin{equation}\label{eq-xi-xii}
	\xi_0^{n+1} =\rho_h^{n+1} + \sum_{i=0}^{s}\alpha_i\tau^i \xi_i^n.
\end{equation}
We now need to simplify $\xi_i^n$ in terms of $\xi_0^n$, which calls for the following definition.
\begin{equation}\label{eq-zeta}
\zeta_i^n = 
	 	\left\{ 
	 	\begin{matrix}
	 	\left(\Pi L -L_h \Pi \right) L^{i-1} u(t^n),& 1\leq i \leq p+1, \\
	 	0, &p+2\leq i\leq s.
	 	\end{matrix}\right. \\
\end{equation}

\begin{LEM}\label{lem-Lxi}
	With $\xi_i^n$ defined in \eqref{eq-xi} and $\zeta_i^n$ defined in \eqref{eq-zeta}, we have
	\begin{equation}\label{eq-Lxi-1}
		\xi_i^n = L_h\xi_{i-1}^n + \zeta_{i}^n, \qquad \forall 1\leq i\leq s. 
	\end{equation}
	As a consequence, 
	\begin{equation}\label{eq-Lxi-ip1}
		\xi_i^n = L_h^i \xi_0^n + \sum_{j=1}^i L_h^{i-j}\zeta_j^n, \qquad \forall 0 \leq i \leq s.
	\end{equation}
\end{LEM}
\begin{proof}
	Using the definition of $\xi_i^n$ in \eqref{eq-xi} and the fact $\zeta_i^n = 0$, $\forall i\geq p+2$, we have 
	\[\xi_i^n = L_h\xi_{i-1}^n = L_h\xi_{i-1}^n + \zeta_{i}^n, \qquad \forall i \geq p+2.\]
	Meanwhile, for $1\leq i\leq p+1$, it can be shown that
	\begin{equation}
	\begin{aligned}
		{\xi_{i}^n}=  \Pi{ L^{i} u(t^n)- L_h^{i} u_h^n} 
		 = &\; \Pi L^{i} u(t^n) - L_h\Pi L^{i-1} u(t^n) + L_h\Pi{ L^{i-1} u(t^n)- L_h^{i} u_h^n} \\
		 = &\; \left(\Pi L - L_h\Pi\right) L^{i-1} u(t^n) + L_h\left(\Pi L^{i-1} u(t^n)- L_h^{i-1}u_h^n \right) \\
		 = &\; \zeta_{i}^n  + L_h \xi_{i-1}^n.
	\end{aligned}
	\end{equation}
	\eqref{eq-Lxi-ip1} can be obtained after repeated application of \eqref{eq-Lxi-1}.
\end{proof}
Substitute \eqref{eq-Lxi-ip1} into \eqref{eq-xi-xii} and it yields that
\begin{equation}
\xi_0^{n+1} = R_s(\tau L_h)\xi_0^n +\left( \sum_{i=0}^s \alpha_i\sum_{j=1}^i \tau^{i} L_h^{i-j}\zeta_j^n\right) + \rho_h^{n+1}.
\end{equation}
Then with the triangle inequality and the fact $\tau^i \nm{ L_h^{i-j}}\leq \tau^i \nm{L_h}^{i-j}\leq \tau^j \lambda^{i-j} <\tau^j$, we have
\begin{equation}
\nm{\xi_0^{n+1}} \leq  \nm{R_s(\tau L_h)}\nm{\xi_0^n} + \left(\sum_{i=0}^s |\alpha_i|\sum_{j=1}^i \tau^{j} \nm{\zeta_j^n}\right) + \nm{\rho_h^{n+1}}.
\end{equation}
By invoking the discrete Gr\"onwall's inequality \eqref{eq-dgronwall} and using the assumption $\nm{R_s(\tau L_h)}\leq 1+\tmu h$ with Remark \ref{rmk-dg}, one can obtain that
\begin{equation}\label{eq-xi0b}
\begin{aligned}
	\nm{\xi_0^{n+1}} \leq &\;\nm{R_s(\tau L_h)}^{n+1}\nm{\xi_0^0} + \left(\sum_{i=0}^n\nm{R_s(\tau L_h)}^i \right)\max_n\left(\left(\sum_{i=0}^s |\alpha_i|\right)\left(\sum_{j=1}^s \tau^{j} \nm{\zeta_j^n}\right) + \nm{\rho_h^{n+1}}\right)\\
	\leq &\;e^{\tmu t^{n+1}}\nm{\xi_0^0} + \sigma(\tmu,t^{n+1})\tau^{-1}\max_n\left(\left(\sum_{i=0}^s |\alpha_i|\right)\left(\sum_{j=1}^{p+1} \tau^{j} \nm{\zeta_j^n}\right) + \nm{\rho_h^{n+1}}\right).
\end{aligned}
\end{equation}
Here we have used the fact $\zeta_j^n = 0$, $\forall j \geq p+2$. 

\begin{LEM}\label{lem-xipirho}
	\[\nm{\zeta_i^n} \leq E_i(h), \qquad \forall 1\leq i\leq p+1.\]
	\[\nm{\rho_h^{n+1}} \leq \left(\sum_{i=0}^s|\alpha_i|\right)\left(\sup_{t} \nm{\partial_{t}^{p+1} u} + E_{p+1}(h)\right)\tau^{p+1}.\]
\end{LEM}

\begin{proof}
	 By adding and subtracting $L^i u(t^n)$ and using the triangle inequality, it yields that
	\begin{equation}
	\nm{\zeta_i^n} \leq \nm{(L-L_h \Pi)L^{i-1}u(t^n)} + \nm{(I-\Pi)L^i u(t^n)}\leq E_i(h). 
	\end{equation}
	Here we have used Assumption \ref{assp-hard} in the last inequality. To estimate $\rho_h^{n+1}$, 
	we denote by 
	\begin{equation}
	\rho_\star^{n+1} = u(t^{n+1})-\sum_{i = 0}^p\alpha_i (\tau L)^i u(t^n) \quand \delta^{n+1} = \left(\sum_{i={p+1}}^s\alpha_i(\tau L_h)^{i-(p+1)} \Pi L^{p+1}u(t^n)\right).
	\end{equation}
	Then 
	\begin{equation}\label{eq-pirho}
	\begin{aligned}
		\rho_h^{n+1} = &\; \Pi \rho_\star^{n+1} - \tau^{p+1} \delta^{n+1} = \rho_\star^{n+1} - (I-\Pi) \rho_\star^{n+1} - \tau^{p+1} \delta^{n+1}.
	\end{aligned}
	\end{equation}
	By estimating each term in \eqref{eq-pirho}, one can obtain the following inequalities.
	\begin{equation}
		\nm{\rho_{\star}^{n+1} }\leq  \frac{1}{(p+1)!}\sup_t\nm{\partial_t^{p+1}u}\tau^{p+1},
	\end{equation}
	\begin{equation}
	\begin{aligned}
	\nm{(I-\Pi) \rho_\star^{n+1}} =&\; \nm{(I-\Pi)\int_{t^n}^t\partial_t^{p+1}u(t') \frac{(t'-t^n)^{p}}{p!} dt'}\\
	= &\;\nm{\int_{t^n}^t (I-\Pi)L^{p+1}u(t') \frac{(t'-t^n)^{p}}{p!} dt'}
	\leq  E_{p+1}(h) \frac{\tau^{p+1}}{(p+1)!},
	\end{aligned}
	\end{equation}
	\begin{equation}
	\begin{aligned}
	\nm{\delta^{n+1}}
	\leq &\; \left(\sum_{i={p+1}}^s|\alpha_i|\lambda^{i-(p+1)}\right)\nm{ \Pi\partial_t^{p+1}u(t^n)}\\
	\leq&\; \left(\sum_{i={p+1}}^s|\alpha_i|\lambda^{i-(p+1)}\right)\left( \nm{\partial_t^{p+1}u(t^n)} + \nm{ (I-\Pi)\partial_t^{p+1}u(t^n)}\right)\\
	\leq&\; \left(\sum_{i={p+1}}^s|\alpha_i|\right)\left( \nm{\partial_t^{p+1}u(t^n)} +  E_{p+1}(h)\right).
	\end{aligned}
	\end{equation}
	Note $\sum_{i=0}^s|\alpha_i|\geq  \frac{1}{(p+1)!}+\sum_{i=p+1}^s|\alpha_i|$. The proof is then completed after applying the triangle inequality in \eqref{eq-pirho} and combining the above estimates.
	
\end{proof}

Applying Lemma \ref{lem-xipirho} to \eqref{eq-xi0b} leads to
\begin{equation}\label{eq-xievol}
\begin{aligned}
\nm{\xi_0^{n+1}} 
\leq &\;e^{\tmu t^{n+1}}\nm{\xi_0^0} + 2\left(\sum_{i = 0}^{s}|\alpha_i|\right) \sigma(\tmu,t^{n+1})\left(\sum_{j=1}^{p+1} \tau^{j-1} E_j(h) +  \sup_t\nm{\partial_t^{p+1} u}\tau^p\right).
\end{aligned}
\end{equation}
We then apply the triangle inequality to obtain
	\begin{equation}\label{eq-eevol}
	\begin{aligned}
\nm{u(t^{n+1}) - u_h^{n+1}} 
	\leq  \nm{(I-\Pi) u(t^0)} + \nm{\xi_0^{n+1}}\leq E_0(h) + \nm{\xi_0^{n+1}}.
	\end{aligned}
	\end{equation}
	The proof of Theorem \ref{thm-main-hard} can be completed by substituting \eqref{eq-xievol} into \eqref{eq-eevol}.

%
%

\section{Applications}\label{sec-app}
\setcounter{equation}{0}
In this section, we apply the error estimates in Section \ref{sec-errest} to various schemes. 
Only error estimates of the fully discrete schemes are considered, since most semidiscrete results can be found in the literature. We apply the following simplifications for ease of presentation.
\begin{itemize}
	\item We assume $u$ (and also $w$) to be a sufficiently smooth function satisfying the given boundary condition.
	\item Initial data is taken as $u_h^0 = \Pi u(0)$.
	\item We verify $\ip{L_h v_h, v_h}\leq \mu\nm{v_h}^2$ instead of $R(\tau L_h)\leq 1 + \mu_h \tau$. (Recall Remark \ref{rem-timestep}.)
\end{itemize}

\subsection{DG and CG schemes for the heat equation}\label{sec-heat}
For parabolic problems, one can choose $\Pi$ to be the elliptic projection for error estimates. For better illustration, let us 
consider the heat equation
\begin{equation}\label{eq-heat}
	u_t = \Delta u
\end{equation}
 with homogeneous Dirichlet boundary conditions.  The classical CG method and stable and consistent DG methods in \cite{arnold2002unified} can be used for the spatial discretization. We recover the notation of bilinear forms and the semidiscrete scheme to \eqref{eq-heat} is given as follows. Find $u_h \in V_h$, such that 
\begin{equation}\label{eq-heatsemi}
	\ip{\partial_t u_h,v_h} + B_h(u_h,v_h) = 0,\qquad  \forall v_h \in V_h.
\end{equation}
For the mentioned methods, the bilinear form $B_h(\cdot,\cdot)$ comes from a stable and consistent discretization of the Poisson equation. 
Therefore, there exists positive constants $C$ and $\nu$, such that
\begin{align}
B_h(w_h,v_h) \leq &\; C\vertiii{w_h}\vertiii{v_h}, \qquad \forall w_h,v_h \in V_h,\\
B_h(v_h,v_h) \geq &\;\nu\vertiii{v_h}^2, \qquad \qquad~ \forall v_h \in V_h.\label{eq-coer}
\end{align}
Here $\vertiii{\cdot}$ is the energy norm. It can then be shown that, for any sufficiently smooth function $w$, the steady state problem
\begin{equation}
	B_h(\Pi w,v_h) = \ip{-\Delta w,v_h} \label{eq-ellproj}
\end{equation}
has a unique solution $\Pi w$ with the error estimates
$\vertiii{w-\Pi w} \leq C\nm{w}_{k+1}h^k$.
By using a standard duality argument, one can obtain
\begin{equation}\label{eq-ellp-errest}
	\nm{w-\Pi w} \leq C\vertiii{w-\Pi w}h \leq C\nm{w}_{k+1}h^{k+1}.
\end{equation}
Note that \eqref{eq-ellproj} can be rewritten in the form of
 $L_h \Pi w = \Pio L w$.
Hence $C_{\rm{S}}\equiv 0$ in Corollary \ref{cor-main-hard} and we can set $k' = 1$ in \eqref{eq-CS}. One can also obtain \eqref{eq-CT} from 
the error estimates of the steady state problem \eqref{eq-ellp-errest}. Semiboundedness of $L_h$ is implied by the coercivity \eqref{eq-coer}. As a result, we have the following error estimate of the fully discrete scheme.
\begin{equation}
	\nm{u(t^{n})-u_h^n}\leq C_{\rm{E}}\left(\sigma(\tmu,t^n)+1\right)\left(h^{k+1} + \tau^p\right).
\end{equation}

\subsection{LDG schemes for 1D equation with high order derivatives}\label{sec-GLDG}

In this section, we consider the (local) DG discretization of the 1D scalar equation
\begin{equation}\label{eq-1deq}
	\partial_t u = \beta \partial_x^q u
\end{equation}
with the periodic boundary condition. Here $\beta$ is a constant. For wellposedness, we assume  $\beta(-1)^\gamma <0$ if $q = 2\gamma$ is even. In particular, our discussion includes the advection equation
\begin{equation}\label{eq-adv}
	\partial_t u + \partial_x u = 0,
\end{equation}
the heat equation
\begin{equation}
\partial_t u = \partial_{xx} u,
\end{equation}
and the dispersive equation
\begin{equation}\label{eq-dispersive}
\partial_t u = \partial_{xxx} u.
\end{equation}

\subsubsection{DG discretization}

We first introduce notations for the DG discretization. Consider a quasi-uniform mesh partition of the domain $\Omega = \cup_{j=1}^N I_j = \cup_{j=1}^N [x_{j-\hf},x_{j+\hf}]$. We denote by $\ip{w,v}_j = \int_{I_j} w v dx$ for $L^2$ inner product on $I_j$.  The finite element space is chosen as follows.
\begin{equation}
V_h  = \{v_h\in L^2(\Omega): v_h|_{I_j} \in P^k(I_j), \forall j\} := V_{h,1}, 
\end{equation}
where $P^k(I_j)$ is the linear space spanned by polynomials of degree no more than $k$. Since $v_h\in V_{h}$ can be discontinuous across cell interfaces, we denote by $v_h^-$ and $v_h^+$ the left and right limits correspondingly. Notations $[v_h] = v_h^+- v_h^-$ and $\{v_h\} = \hf(v_h^+ + v_h^-)$ are used to represent jumps and averages. 

The DG operator $D_{h,\theta}$ for approximating $\partial_x$ is defined through the variational form
\begin{equation}\label{eq-DG-1D}
\ip{D_{h,\theta} w_h,v_h } = -\ip{ u_h, \partial_xv_h}  - \sum_{j} \widehat{ w_h}_{j+\hf}[v_h]_{j+\hf} ,\qquad \forall w_h, v_h \in V_{h},
\end{equation}
where the numerical flux is 
\begin{equation}\label{eq-DG-1Dflux}
\widehat{w_h} =\theta w_h^- + (1-\theta)w_h^+.
\end{equation}
In particular, depending on the sign of $\beta$, the upwind and downwind fluxes can be retrieved with $\theta = 0$ and $\theta = 1$. The case $\theta = \hf$ corresponds to the central flux.
 
One can verify the following property of $D_{h,\theta}$.

\begin{PROP}[Antisymmetry]\label{prop-antisym}
$D_{h,\theta}^\top = -D_{h,1-\theta}$. 
To be more specific, we have 
\begin{equation}
\ip{D_{h,\theta}  w_h, v_h} = -\ip{w_h,D_{h,1-\theta}v_h},\qquad \forall w_h, v_h \in V_h.
\end{equation}
\end{PROP}
\begin{PROP}[Semidefiniteness]
	$\ip{D_{h,\theta}v_h,v_h} = \left(\theta-\hf\right)\sum_j [v_h]_{j+\hf}^2$.
\end{PROP}

The DG scheme for \eqref{eq-1deq} can be obtained by replacing $\partial_x$ with $D_{h,\theta}$. $\theta$ has to be appropriately chosen to ensure stability. To be more specific, we take
\begin{equation}\label{eq-LhLDG}
\partial_t u_h = L_h u_h, \qquad
L_h = \left\{
\begin{array}{cc}
\beta K_h^\top K_h, &
q = 2\gamma,\\
\beta K_h^\top D_{h,\theta_0} K_h,&
q = 2\gamma + 1,
\end{array}\right.
\end{equation}
where $\theta_0$ is a constant such that $\beta(\theta_0 -\hf) \leq 0$ and 
$K_h = D_{h,\theta_1}D_{h,\theta_2}\cdots D_{h,\theta_\gamma}$, $\forall \theta_1, \theta_2, \cdots, \theta_\gamma$.
Also note that
\begin{equation}
K_h^\top = (-1)^{\gamma}D_{h,1-\theta_\gamma}\cdots D_{h,1-\theta_2}D_{h,1-\theta_1}.
\end{equation}
In particular, we have the follow semidiscrete schemes for the advection equation, the heat equation and the dispersive equation correspondingly. 
	\begin{align}
	\partial_t u_h =&\; -D_{h,\theta_0} u_h , \qquad \theta_0 \geq \hf.\label{eq-dg-adv}\\
		\partial_t u_h =&\;  -D_{h,1-\theta_1}D_{h,\theta_1} u_h , \qquad \forall \theta_1.\\
	 \partial_t u_h =&\;  -D_{h,1-\theta_1}D_{h,\theta_0}D_{h,\theta_1} u_h , \qquad \theta_0\leq \hf, \quad  \forall \theta_1.
	 \end{align}

\begin{PROP}\label{eq-dg-stab}
	For $L_h$ defined in \eqref{eq-LhLDG}, we have $\ip{L_h v_h, v_h} \leq 0 $.
\end{PROP}

The remaining task is to construct the operator $\Pi$. We assume $\theta_i \neq \hf$, $\forall i$, since the convergence rate may degenerate in this case. 

\subsubsection{Advection equation}
Before going into the general equation \eqref{eq-1deq}, we first consider the advection equation, with $L = - \partial_x$  and $L_h = - D_{h,\theta}$. This scheme has been studied in \cite{meng2016optimal}. The corresponding projection operator was first constructed in \cite[Lemma 2.6]{meng2016optimal}. Then in \cite[Lemma 3.2]{cheng2017application}, Cheng et al. reduced the regularity assumptions in the approximation results.

\begin{LEM}
	\label{lem-MSW}
	For any $\theta \neq \hf$, there exists a uniquely defined $\Pi_\theta w \in  V_{h} $ such that
	\begin{align}
	\ip{\Pi_\theta w, v_h}_j = 		\ip{w, v_h}_j , \qquad &\forall v_h \in P^{k-1}(I_j),\\
	\widehat{\Pi_\theta w}  = \widehat{w}, ~ \qquad  \qquad&\text{at }x = x_{j+\hf}, \qquad \forall j.
	\end{align}
	Here $\widehat{\Pi_\theta w}$ is given in \eqref{eq-DG-1Dflux} and $\widehat{w}=\theta w^- + (1-\theta) w^+$.\footnote{Similar conventions are used in the rest of the paper. We state the definition of numerical fluxes for functions in $V_h$, while similar notations also apply to $w$ as well. Repeated definitions are omitted.} Furthermore, we have
	$\nm{w - \Pi_\theta w} \leq C\nm{w}_ih^{i}$, $\forall 1\leq i\leq k+1$. 
\end{LEM}

By rewriting \cite[Lemma 2.8]{meng2016optimal} in the operator form, we have
\begin{PROP}\label{prop-abs}
		$D_{h,\theta} \Pi_\theta w = \Pio \partial_xw$, for any $w$ that is periodic and absolutely continuous.
\end{PROP}

As a result, we have $C_{\rm{S}} \equiv 0$ and $k' = 1$ in Corollary \ref{cor-err}. The required approximation property has also been verified in Lemma \ref{lem-MSW}. Therefore, after applying the RK time discretization,  the fully discrete DG scheme based on \eqref{eq-dg-adv} has the error estimate
\begin{equation}
\nm{u(t^{n})-u_h^n}\leq C_{\rm{E}}\left(\sigma(\tmu,t^n)+1\right)\left(h^{k+1} + \tau^p\right).
\end{equation}

\subsubsection{Equations with high order derivatives}\label{sec-xu}
Error estimates of the LDG methods for time-dependent equations with high order spatial derivatives are usually based on the mixed form. Projections are constructed not only for $u$ but also for auxiliary unknowns, which can not be directly applied in our framework. In this section, we construct $\Pi u$ that can be used for the primal form \eqref{eq-fully}. 

Discussions in Section \ref{sec-heat} indicates that, one way of constructing $\Pi$ is to formally solve the steady state problem, and set
$
	\Pi w = L_h^{-1} \Pio L w
$. 
Since $L_h$ is constructed as compositions of $D_{h,\theta_i}$, it motivates us to investigate the inverse of $D_{h,\theta}$. While to have $D_{h,\theta}^{-1}$ well defined, we need to look into 
a suitable subspace of $V_h$. To be more specific, we will show that $D_{h,\theta}$ is invertible on
\begin{equation}
Z_h = V_h \cap Z := V_h \cap \{z \in L^2(\Omega): \ip{z,1} = 0\}.
\end{equation}

The inverse of $D_{h,0}$ and $D_{h,1}$ have been discussed by Ji and Xu in \cite[Appendix A3]{ji2011optimal}, in the 2D context for analyzing the LDG method for Willmore flow. Here we reinterpret the 1D case with any $\theta \neq \hf$. Such inverse operator also relates to those used in the superconvergence analysis \cite{xu2019superconvergence}. 
\begin{PROP}\label{prop-ellp}
	Given $z_h \in Z_h$ and $\theta \neq \hf$, we have
	\begin{equation}\label{eq-wid}
	\nm{z_h}^2 = -\ip{D_{h,\theta}z_h, \Pi_{1-\theta}\zeta_z},
	\end{equation}
	where 
	$\zeta_z = \int_{x_\hf}^x z_h(x) dx$ 
	and $\Pi_{1-\theta}$ is defined in Lemma \ref{lem-MSW}.
\end{PROP}
	\begin{proof}
	Note that $\zeta_z$ is indeed periodically define, because $\zeta_z(x_{N+\hf}) = \ip{z_h,1} = 0 = \zeta_z(x_\hf)$. 
	It is also absolutely continuous by definition. Since $\Pio \partial_x \zeta_z= D_{h,1-\theta}\Pi_{1-\theta} \zeta_z$ (Proposition \ref{prop-abs}) and $D_{h,1-\theta}^\top = -D_{h,\theta}$ (Proposition \ref{prop-antisym}), it can be shown that
	\begin{equation}
	\begin{aligned}
	\nm{z_h}^2 = &\;\ip{z_h, \partial_x\zeta_z}  = \ip{z_h,\Pio \partial_x \zeta_z} = \ip{z_h,D_{h,1-\theta}\Pi_{1-\theta} \zeta_z}\\
	=  &\;
	\ip{D_{h,1-\theta}^\top z_h,\Pi_{1-\theta} \zeta_z} = 
	-\ip{D_{h,\theta}z_h,\Pi_{1-\theta} \zeta_z}.
	\end{aligned}
	\end{equation}
\end{proof}

\begin{PROP}\label{prop-bi}
	Suppose $\theta\neq \hf$. Then $D_{h,\theta}: Z_h\to Z_h$ is a bijection, and 
	\begin{equation}
	\nm{D_{h,\theta}^{-1}z_h} \leq C \nm{z_h}, \qquad \forall z_h \in Z_h,\qquad \text{for some constant } C.
	\end{equation}
\end{PROP}
\begin{proof}
	Since
	$\ip{ D_{h,\theta} z_h ,1}= \sum_j \left(-\ip{ z_h, \partial_x 1}_j + \widehat{z_h}_{j+\hf} - \widehat{z_h}_{j-\hf} \right) = 0$,
	we have $D_{h,\theta} z_h \in Z_h$, which implies $\mathrm{Range}(D_{h,\theta})\subset Z_h$.
	To prove $D_{h,\theta}$ is a bijection, it suffices to verify it is an injection due to the finite dimensionality of $Z_h$. In other words, we need to show that
	\begin{equation}
	D_{h,\theta} z_h = 0, \quad z_h\in Z_h \quad \Rightarrow \quad z_h = 0.
	\end{equation}	 
	Indeed, this can be proved with Proposition \ref{prop-bi} by noting 
	$
	\nm{z_h}^2 = -\ip{D_{h,\theta}z_h, \Pi_{1-\theta}\zeta_z} = 0
	$.
	
	To estimate the bound of $\nm{D_{h,\theta}^{-1}z_h}$, we once again apply Proposition \ref{prop-ellp} with $z_h$ replaced with $v_h = D_{h,\theta}^{-1} z_h$. Then we have
	\begin{equation}
	\begin{aligned}
	\nm{v_h}^2 =&\; -\ip{D_{h,\theta}v_h, \Pi_{1-\theta}\zeta_v}= -\ip{z_h,\Pi_{1-\theta}\zeta_v} \leq \nm{z_h}\nm{\Pi_{1-\theta} \zeta_v}\\
	\leq &\;\nm{z_h}\left(\nm{\zeta_v} + \nm{(I-\Pi_{1-\theta}) \zeta_v}\right)\leq C\nm{z_h}\left(\nm{\zeta_v} +  h \nm{\zeta_v}_1\right) \leq C \nm{z_h}\nm{v_h}.
	\end{aligned}
	\end{equation}
	Here we have used the approximation property of $\Pi_{1-\theta}$ in Lemma \ref{lem-MSW} with $i = 1$ and the definition of $\zeta_v$. The proof is then completed after dividing by $\nm{v_h}$ on both sides. 
\end{proof}

We prove the following properties for compositions of $D_{h,\theta_i}^{-1}$. 

\begin{LEM}\label{lem-ldgapprox-1} 
	 Suppose $w\in Z$ is periodic and sufficiently smooth. If $\theta_j\neq \hf$, $\forall 1\leq j \leq i$, then we have
	\begin{equation}\label{eq-spproj}
	\nm{w -  D_{h,\theta_i}^{-1}D_{h,\theta_{i-1}}^{-1}\cdots  D_{h,\theta_1}^{-1}\Pio\partial_x^i w}\leq C\nm{w}_{k+1+i}h^{k+1}.
	\end{equation}
\end{LEM}
\begin{proof}
	Let us denote by $G_{h,i} = D_{h,\theta_i}^{-1}D_{h,\theta_{i-1}}^{-1}\cdots  D_{h,\theta_1}^{-1}$. First, note that
	\begin{equation}
		\ip{\Pio \partial_x^i w, 1} = 	\ip{\partial_x^i w, 1} = \partial_x^{i-1} w \big|_{x_{\hf}}^{x_{N+\hf}} = 0, \qquad \forall i \geq 1. 
	\end{equation}
	Therefore, $\Pio\partial_x^i w \in Z_h$, $\forall i \geq 1$ and $G_{h,i}\Pio\partial_x^i w$ is well-defined for all $i \geq0$.
	We then prove \eqref{eq-spproj} by induction. The case $i = 0$ corresponds to 
	$\nm{w - \Pio w} \leq C\nm{w}_{k+1}h^{k+1}$,
	which is simply the approximation property of $\Pio$. Suppose \eqref{eq-spproj} is true for all integers no larger than $i$. Then with the triangle inequality, we have
	\begin{equation}\label{eq-spest}
	\begin{aligned}
	\nm{w-G_{h,i+1}\Pio \partial_x^{i+1}w} 
	\leq  \nm{w-D_{h,\theta_{i+1}}^{-1}\Pio \partial_x w} + \nm{D_{h,\theta_{i+1}}^{-1}\left(\Pio \partial_x w- G_{h,i} \Pio \partial_x^{i+1}w\right)} 
	\end{aligned}
	\end{equation}
	We start with estimating the first term. The main technicality is to include the case $k = 0$. Note that $ \Pi_{\theta_{i+1}} w - \frac{\ip{\Pi_{\theta_{i+1}} w, 1}}{\ip{1,1}} \in Z_h$ and $D_{h,{\theta_{i+1}}} \left(\Pi_{\theta_{i+1}} w -  \frac{\ip{\Pi_{\theta_{i+1}} w, 1}}{\ip{1,1}}\right) = \Pio \partial_x w$. Since $D_{h,{\theta_{i+1}}}^{-1}$ is uniquely defined on $Z_h$, we have $D_{\theta_{i+1}}^{-1} \Pio \partial_x w =  \Pi_{\theta_{i+1}} w -  \frac{\ip{\Pi_{\theta_{i+1}} w, 1}}{\ip{1,1}} $. Recall that $\ip{w,1} = 0$, and it can be shown that
	\begin{equation}
	\begin{aligned}
		\nm{w-D_{h,\theta_{i+1}}^{-1}\Pio \partial_x w} 
		=&\;\nm{w - \left(\Pi_{\theta_{i+1}}w - \frac{\ip{\Pi_{\theta_{i+1}}w-w,1}}{\ip{1,1}}\right)}\\
		\leq &\;\nm{w-\Pi_{\theta_{i+1}} w} + \nm{\frac{\ip{\Pi_{\theta_{i+1}} w-w, 1}}{\ip{1,1}} } \\
		\leq &\;	\nm{w-\Pi_{\theta_{i+1}} w} + {\frac{|\ip{\Pi_{\theta_{i+1}} w -w, 1}|}{\sqrt{\ip{1,1}} }}\\
		\leq &\; 2\nm{w-\Pi_{\theta_{i+1}} w}.
	\end{aligned}
	\end{equation}
	Here we have applied the H\"older's inequality in the last step.
	Using the approximation property of $\Pi_{\theta_{i+1}}$ in Lemma \ref{lem-MSW}, we have
	\begin{equation}\label{eq-spest-2}
	\nm{w-D_{h,\theta_{i+1}}^{-1}\Pio \partial_x w}  \leq C\nm{w}_{k+1}h^{k+1}.
	\end{equation}
	For the second term in \eqref{eq-spest}, one can apply Proposition \ref{prop-bi} to obtain
	\begin{equation}\label{eq-spest-3}
	\begin{aligned}
	\nm{D_{h,\theta_{i+1}}^{-1}\left(\Pio \partial_x w- G_{h,i}\Pio \partial_x^{i+1}w\right)}
	\leq &\; C\nm{\Pio \partial_x w- G_{h,i}\Pio \partial_x^{i+1}w}\\
	\leq &\; C\left(\nm{\partial_x w - \Pio \partial_x w}+\nm{\partial_x w - G_{h,i}\Pio \partial_x^{i}(\partial_x w)}\right)\\
	\leq &\;
	C\left(\nm{w}_{k+2}h^{k+1}+\nm{w}_{k+1+i+1}h^{k+1}\right).
	\\
	\leq &\;
	C\nm{w}_{k+2+i}h^{k+1}.
	\end{aligned}
	\end{equation}
	Here we have used the induction assumption and the approximation property of $\Pio$ (stated later in Proposition \ref{prop-L2}) in the third inequality  of \eqref{eq-spest-3}. Finally, by substituting \eqref{eq-spest-2} and \eqref{eq-spest-3} into \eqref{eq-spest}, it can be shown that
	\begin{equation}
	\nm{w-G_{h,\theta_{i+1}}\Pio \partial_x^{i+1}w} \leq C\nm{w}_{k+2+i}h^{k+1}, 
	\end{equation}
	which completes the proof. 
\end{proof}

\begin{THM}\label{lem-ldg-proj}
	Let 
	\begin{equation}
	\Pi w = 
	\left\{
	\begin{array}{ll}
	\left(D_{h,\theta_{\gamma}}^{-1} \cdots D_{h,\theta_1}^{-1}\right)\left(D_{h,1-\theta_1}^{-1}\cdots D_{h,1-\theta_\gamma}^{-1} \right)\Pio \partial_x^q w + \frac{\ip{w,1}}{\ip{1,1}}, & q = 2\gamma,\\
	\left(D_{h,\theta_{\gamma}}^{-1} \cdots D_{h,\theta_1}^{-1}\right)D_{h,\theta_0}^{-1}\left(D_{h,1-\theta_1}^{-1}\cdots D_{h,1-\theta_\gamma}^{-1}\right) \Pio \partial_x^q w + \frac{\ip{w,1}}{\ip{1,1}}, & q = 2\gamma+1.
	\end{array}\right.
	\end{equation}
	Then 
	$\nm{w - \Pi w} \leq C \nm{w}_{k+1+q}h^{k+1}$.
\end{THM}
According to the construction of $\Pi$, it can be verified that $L_h \Pi w = \Pio L w$. Hence we have $C_{\rm{S}} = 0$, $k' = 1$ and $C_{\rm{T}}$ is a constant dependent on $\nm{w}_{k+1+q}$. Therefore, for $\theta_i \neq \hf$, $\forall i$, the LDG scheme \eqref{eq-LhLDG} with an explicit RK time discretization has the error estimate
\begin{equation}
\nm{u(t^{n})-u_h^n}\leq C_{\rm{E}}\left(\sigma(\tmu,t^n)+1\right)\left(h^{k+1} + \tau^p\right).
\end{equation}

\begin{REM}
	When $k\geq 1$, by recalling the definition of $\Pi_\theta$ in Lemma \ref{lem-MSW}, we have $\Pi_{1-\theta_\gamma} \partial_x^{q-1}w \in Z_h$, $\forall q \geq 2$. Therefore,  $D_{1-\theta_{\gamma}}\Pi_{1-\theta_\gamma} \partial_x^{q-1} w = \Pi_0 \partial_x^q w$ (Proposition \ref{prop-abs}) implies that 
	$D_{1-\theta_{\gamma}}^{-1}  \Pi_0 \partial_x^q w = \Pi_{1-\theta_\gamma} \partial_x^{q-1} w$. Then Lemma \ref{lem-ldg-proj} can be rephrased as follows.
	\begin{equation}
	\Pi w = 
	\left\{
	\begin{array}{ll}
	D_{h,\theta_1}^{-1}\Pi_{1-\theta_1} \partial_x w + \frac{\ip{w,1}}{\ip{1,1}}, &q = 2,\\
	D_{h,\theta_1}^{-1}D_{h,\theta_0}^{-1}\Pi_{1-\theta_1} \partial_x^{2} w + \frac{\ip{w,1}}{\ip{1,1}}, &q = 3,\\
	\left(D_{h,\theta_{\gamma}}^{-1} \cdots D_{h,\theta_1}^{-1}\right)\left(D_{h,1-\theta_1}^{-1}\cdots D_{h,1-\theta_{\gamma-1}}^{-1} \right)\Pi_{1-\theta_{\gamma}} \partial_x^{q-1} w + \frac{\ip{w,1}}{\ip{1,1}}, & q = 2\gamma,\gamma>1,\\
	\left(D_{h,\theta_{\gamma}}^{-1} \cdots D_{h,\theta_1}^{-1}\right)D_{h,\theta_0}^{-1}\left(D_{h,1-\theta_1}^{-1}\cdots D_{h,1-\theta_{\gamma-1}}^{-1}\right) \Pi_{1-\theta_{\gamma}} \partial_x^{q-1} w + \frac{\ip{w,1}}{\ip{1,1}}, & q = 2\gamma+1, \gamma>1.
	\end{array}\right.
	\end{equation}
	Furthermore, $\nm{w - \Pi w} \leq C \nm{w}_{k+q}h^{k+1}$. In the case $q = 3$, $\Pi u(0)$ would retrieve the initial condition specified in \cite[page 86]{xu2012optimal}.
\end{REM}

\subsection{DG schemes for 1D scalar equations}

We adopt notations in Section \ref{sec-GLDG} in the following examples. 

\begin{examp}[DG method for advection equation with central flux]\label{examp-MSW}
\;
\upshape
Consider the scheme \eqref{eq-dg-adv} with $\theta = \hf$ for solving the advection equation \eqref{eq-adv}. Semiboundedness of $L_h = D_{h,\hf}$ has been verified in Proposition \ref{eq-dg-stab}. We set $\Pi = \Pio$ for the error estimate. 
\begin{PROP}\label{prop-L2}
	Let $\Pio$ be the $L^2$ projection to $V_{h}$.  Then 
	\begin{equation}
	\ip{\Pio w, v_h}_j = 		\ip{w, v_h}_j , \qquad \forall v_h  \in P^{k}(I_j), \qquad\forall j,
	\end{equation}
	and we have
	\begin{equation}\label{eq-L2-proj}
	\nm{w - \Pio w } + \sqrt{h\sum_j \left(w -( \Pio  w)^\pm \right)_{j+\hf}^2 }\leq Ch^{k+1}.
	\end{equation}
	Here $C$ depends on the $(k+1)$th Sobolev norm of $w$.
\end{PROP}
Then it can be shown that
\begin{equation}
\begin{aligned}
	\ip{L_h\Pi w - L w,v_h} = &\; \sum_j \left(\widehat{\Pi w} - w\right)_{j+\hf} [v_h]_{j+\hf} \\
	\leq &\;  \sqrt{\hf \sum_j \left(\Pi w^--w\right)_{j+\hf}^2 +  \hf \sum_j \left(\Pi w^+- w\right)_{j+\hf}^2} \sqrt{\sum_j[v_h]^2_{j+\hf}}
	\leq  Ch^{k} \nm{v_h}.
	\end{aligned}
\end{equation}
Here we have used the approximation property \eqref{eq-L2-proj} and the inverse estimate $\sqrt{\sum_j[v_h]^2_{j+\hf}} \leq Ch^{-\hf} \nm{v_h}$ in the last step. Hence \eqref{eq-CS} holds with $k' = 0$.  \eqref{eq-CT} is implied by Proposition \ref{prop-L2}. 
Therefore, we have the following error estimate of the fully discrete scheme. 
\begin{equation}
\nm{u(t^{n})-u_h^n}\leq C_{\rm{E}}\left(\sigma(\tmu,t^n)+1\right)\left(h^{k} + \tau^p\right).
\end{equation}
The order degeneration with $\theta = \hf$ can also be observed numerically.
\end{examp}

\begin{examp}[Ultra-weak DG method for dispersive equation]\label{examp-CS3}
	\upshape
	We then consider the ultra-weak DG method \cite{cheng2008discontinuous} for the dispersive equation \eqref{eq-dispersive}.	The discrete operator $L_h$ is defined such that
	\begin{equation}\label{eq-DG-CS3}
	\ip{L_h w_h,v_h} = -\ip{w_h,\partial_{xxx}v_h} -\sum_j\left(\widehat{w_h}[v_h]-\widetilde{ (w_h)_x}[v_h]+\widecheck{ (w_h)_{xx}}[v_h]\right)_{j+\hf}, \forall w_h,v_h \in V_{h},
	\end{equation}
	where the numerical fluxes are 
	\begin{subequations}
		\begin{align}
		\widehat{w_h} = w_h^+, \qquad \widetilde{ (w_h)_x} = (w_h)_x^+, \qquad \widecheck{ (w_h)_{xx}} = (w_h)_{xx}^-.
		\end{align}
	\end{subequations}
	It has been verified in \cite[Section 3.1]{cheng2008discontinuous} that $L_h$ is semibounded with $\mu = 0$. For optimal error estimates, one has to apply the projection introduced in \cite[Section 2.4]{cheng2008discontinuous}.
	\begin{PROP}\label{prop-CS}
		Let $ k \geq 3$. There exists a uniquely defined $\Pi w \in V_{h}$ such that
		\begin{subequations}\label{eq-proj-CS3}
			\begin{align}
			\ip{\Pi w, v_h}_j =		\ip{w, v_h}_j , \qquad & \forall v_h\in P^{k-3}(I_j),\\
			\widehat{\Pi w}  =\widehat{w},\quad  \widetilde{ (\Pi w)_x}  = \widetilde{w_x} , \quad \widecheck{(\Pi w)_{xx}} = \widecheck{w_{xx}}, \quad\quad& \text{at } x = x_{j+\hf}, \quad \forall j.
			\end{align}
		\end{subequations}
		Furthermore, $\nm{w - \Pi w} \leq Ch^{k+1}$, where $C$ depends on $(k+1)$th order Sobolev norm of $w$. 
	\end{PROP}
	With $\Pi$ defined in \eqref{eq-proj-CS3}, It has been shown in \cite[Appendix A.2]{cheng2008discontinuous} that $L_h \Pi w =  \Pio L w$ for sufficiently smooth $w$. Hence $C_{\rm{S}} = 0$ and $k' = 1$. Moreover, \eqref{eq-CT} follows from the approximation property of $\Pi$. Therefore, we have
	\begin{equation}
	\nm{u(t^{n})-u_h^n}\leq C_{\rm{E}}\left(\sigma(\tmu,t^n)+1\right)\left(h^{k+1} + \tau^p\right).
	\end{equation}
\end{examp}

\subsection{DG schemes for 1D ``systems"}

In this section, we consider a class of schemes that solves a scalar equation by introducing auxiliary variables and rewriting the equation into a system. With these examples, one can get a glance on how the framework can be applied to equation systems. In the following examples, $V=	L^2(\Omega;\mathbb{R}^2)$. $V_h$ and $\Pi$ will be redefined for each scheme. 

\begin{examp}[DG method for wave equation with $\alpha\beta$-fluxes]\label{examp-CCLX}
	\upshape
	One way of solving the wave equation 	
	\begin{equation}
	\partial_{tt} u = \partial_{xx} u,
	\end{equation} 
	is to first rewrite the equation into a first-order system 
	\begin{equation}\label{eq-wave}
	\partial_t \left(
	\begin{matrix}
	u\\
	\phi
	\end{matrix}\right) + \partial_x A\left(		\begin{matrix}
	u\\
	\phi
	\end{matrix}\right) = 0,\qquad A = \left(\begin{matrix}
	0&1\\
	1&0 
	\end{matrix}\right) ,
	\end{equation}
	and then apply the DG discretization with suitable numerical fluxes. This scheme has been studied in \cite{cheng2017l2}. 
	The finite element space is taken as $V_h = [V_{h,1}]^2$ and the associated inner-product is defined through
	\begin{equation}
		\ip{\cdot,\cdot} = \ip{\cdot,\cdot}_j, \quad \ip{ \left(
			\begin{matrix}
			w\\
			\chi
			\end{matrix}\right), \left(
			\begin{matrix}
			v\\
			\psi
			\end{matrix}\right)}_j = \int_{I_j} wv + \chi \psi dx.
	\end{equation}	
	The discrete operator $L_h$ for approximating $L = - \partial_x A$ is given by
	\begin{equation}\label{eq-Lh-FS}
	\ip{L_h \left(
		\begin{matrix}
		w_h\\
		\chi_h
		\end{matrix}\right), \left(
		\begin{matrix}
		v_h\\
		\psi_h
		\end{matrix}\right)} = \ip{A\left(
		\begin{matrix}
		w_h\\
		\chi_h
		\end{matrix}\right), \partial_x\left(
		\begin{matrix}
		v_h\\
		\psi_h
		\end{matrix}\right)} + \sum_j  \left(A\widehat{\left(
		\begin{matrix}
		w_h\\
		\chi_h
		\end{matrix}\right)}\cdot  \left(
	\begin{matrix}
	\ [v_h] \\
	\ [\psi_h]
	\end{matrix}\right)
	\right)_{j+\hf},\quad \forall \left(
	\begin{matrix}
	w_h\\
	\chi_h
	\end{matrix}\right), \left(
	\begin{matrix}
	\ v_h \\
	\ \psi_h
	\end{matrix}\right) \in V_h.
	\end{equation}
	where 
	\begin{equation}
	\widehat{\left(
		\begin{matrix}
		w_h\\
		\chi_h
		\end{matrix}\right) } = {\left(
		\begin{matrix}
		\{w_h\} + \alpha[w_h] + \beta_1 [\chi_h]\\
		\{\chi_h\} -\alpha[\chi_h]+\beta_2 [w_h]
		\end{matrix}\right)}, \qquad \beta_1, \beta_2\leq 0.
	\end{equation}
	Furthermore, when $\alpha^2 + \beta_1\beta_2 = \frac{1}{4}$, such fluxes are referred to as $\alpha\beta$-fluxes. It has been shown in \cite[Theorem 2.2]{cheng2017l2} that $L_h$ is semibounded with $\mu = 0$ when $\beta_1, \beta_2 \leq 0$. For error estimates, the following operator is constructed (rephrased from \cite[Lemma 2.4]{cheng2017l2}).
	\begin{PROP}\label{prop-CCLX}
		There exists a uniquely defined $\Pi \left(
		\begin{matrix}
		w\\
		\chi
		\end{matrix}\right)\in V_h$, such that 
		\begin{align}
		\ip{\Pi \left(
			\begin{matrix}
			w\\
			\chi
			\end{matrix}\right), \left(
			\begin{matrix}
			v_h\\
			\psi_h
			\end{matrix}\right)}_j = 	&\;
		\ip{\left(
			\begin{matrix}
			w\\
			\chi
			\end{matrix}\right), \left(
			\begin{matrix}
			v_h\\
			\psi_h
			\end{matrix}\right)}_j,\qquad \forall  v_h, \psi_h \in P^{k-1}(I_j),\\
		\widehat{ \Pi \left(
			\begin{matrix}
			w\\
			\chi
			\end{matrix}\right)} =&\;  \widehat{\left(
		\begin{matrix}
		w\\
		\chi
		\end{matrix}\right)},\qquad\qquad \qquad ~\text{at } x = x_{j+\hf}, \quad \forall j.
		\end{align}
		Furthermore, 
		$\nm{\left(
			\begin{matrix}
			w\\
			\chi
			\end{matrix}\right) -\Pi \left(
			\begin{matrix}
			w\\
			\chi
			\end{matrix}\right) } \leq Ch^{k+1}
		$,
		where $C$ is a constant dependent on the $(k+1)$th order Sobolev norm of $w$ and $\chi$.
		\end{PROP}
		By looking into the proof of \cite[Theorem 2.5, Theorem 2.6]{cheng2017l2}, it can be shown that
		\begin{equation}
		\ip{L_h \Pi \left(
			\begin{matrix}
			w\\
			\chi
			\end{matrix}\right) - L\left(
			\begin{matrix}
			w\\
			\chi
			\end{matrix}\right) ,\left(
			\begin{matrix}
			v_h\\
			\psi_h
			\end{matrix}\right)} = 0, \qquad \alpha^2 + \beta_1\beta_2 = \frac{1}{4},
		\end{equation}
		\begin{equation}
		\left|\ip{L_h \Pi \left(
			\begin{matrix}
			w\\
			\chi
			\end{matrix}\right) - L\left(
			\begin{matrix}
			w\\
			\chi
			\end{matrix}\right) ,\left(
			\begin{matrix}
			v_h\\
			\psi_h
			\end{matrix}\right)} \right|\leq C h^{k+\delta} \nm{\left(
			\begin{matrix}
			v_h\\
			\psi_h
			\end{matrix}\right)}, \qquad \alpha^2 + \beta_1\beta_2 = \frac{1}{4} + C h^\delta.
		\end{equation}
	Therefore, we obtain the error estimate for the fully discrete scheme. 
	\begin{equation}\label{eq-errCCLX}
	\nm{ \left(
		\begin{matrix}
		u(t^n)\\
		\phi(t^n)
		\end{matrix}\right) - \left(
		\begin{matrix}
		u_h^n\\
		\phi_h^n
		\end{matrix}\right) }\leq 
	\left\{
	\begin{matrix}
		 C_{\rm{E}}\left(\sigma(\tmu,t^n)+1\right)\left(h^{k+1} + \tau^p \right), & \alpha^2 + \beta_1 \beta_2 = \frac{1}{4},\\
		 C_{\rm{E}}\left(\sigma(\tmu,t^n)+1\right)\left(h^{k+\min(\delta,1)} + \tau^p \right), & \alpha^2 + \beta_1 \beta_2 = \frac{1}{4} + Ch^\delta .
	\end{matrix}\right.
	\end{equation}
	
	
\end{examp}

\begin{examp}[Energy conserving DG method for conservation laws]\label{examp-FS}
	\upshape
	This example comes from \cite{fu2019optimal}. To solve the advection equation \eqref{eq-adv}, the authors introduced an auxiliary unknown $\phi = 0$ and solved the following augmented system \eqref{eq-FuShu} with the DG method
	\begin{equation}\label{eq-FuShu}
	\partial_t \left(
	\begin{matrix}
	u\\
	\phi
	\end{matrix}\right) + \partial_x B \left(		\begin{matrix}
	u\\
	\phi
	\end{matrix}\right) = 0, \qquad B = \left(\begin{matrix}
	1&0\\
	0&-1 
	\end{matrix}\right).
	\end{equation}
	This scheme achieves optimal convergence rate while conserving the total energy $\int_\Omega u_h^2 + \phi_h^2 dx$ ($\ip{L_hv_h,v_h} = 0, \forall v_h \in V_h$. See \cite[Corollary 2.4]{fu2019optimal}). The setting of the scheme is similar to that in Example \ref{examp-CCLX}, except for replacing the matrix $A$ with $B$ and requiring $\alpha = 0$, $\beta_1 = \beta_2 = \hf$. Following arguments in Example \ref{examp-CCLX}, the fully discrete scheme has the error estimate.
	\begin{equation}
			\nm{u(t^n) - u_h^n}\leq \nm{ \left(
			\begin{matrix}
			u(t^n)\\
			0
			\end{matrix}\right) - \left(
			\begin{matrix}
			u_h^n\\
			\phi_h^n
			\end{matrix}\right) }\leq  C_{\rm{E}}\left(\sigma(\tmu,t^n)+1\right)\left(h^{k+1} + \tau^p \right).
	\end{equation}
	
	In \cite{fu2019optimal}, the authors also considered hyperbolic symmetric systems. The main idea is to use the eigendecomposition to decouple the system into scalar equations, and then introduce auxiliary unknowns to pair up the equations. One can use the projection in Proposition \ref{prop-CCLX} for each pair to show the optimal convergence. Details are omitted. 
	
\end{examp}

\begin{examp}[Central DG method for advection equation]
	\upshape
The central DG method was proposed by Liu et al. \cite{liu2007central} for solving hyperbolic conservation laws. Its application to \eqref{eq-adv} can be considered as applying the DG discretization to the system \eqref{eq-cdg} with an auxiliary unknown $\phi = u$  on overlapping meshes.
\begin{equation}\label{eq-cdg}
	\partial_t \left(\begin{matrix}
	u\\
	\phi
	\end{matrix}\right) + \partial_x A\left(\begin{matrix}
	u\\
	\phi
	\end{matrix}\right) = \frac{1}{\tau_{\max}}\left(\begin{matrix}\phi - u\\u - \phi\end{matrix}\right).
\end{equation}
Here $A$ is defined in \eqref{eq-wave} and $\tau_{\max}= \lambda h$ is the maximum time step size. We denote by $I_{j+\hf} = [x_j, x_{j+1}]$ and 
\begin{equation}
\widetilde{V}_{h,1} = \{\psi_h \in L^2(\Omega): \psi_h|_{I_{j+\hf}} \in P^k(I_{j+\hf}), \forall j \}.
\end{equation} 
Then the finite element space for the central DG method is given by
\begin{equation}
	V_h = V_{h,1}\times \widetilde{V}_{h,1} = \left\{\left(\begin{matrix}
	v_h\\
	\psi_h
	\end{matrix}\right): v_h|_{j} \in P^k(I_j), \psi_h|_{I_{j+\hf}}\in P^k(I_{j+\hf})\right\}.
\end{equation}
With 
\begin{equation}
\ip{ \left(\begin{matrix}
w_h\\
\chi_h
\end{matrix}\right), \left(\begin{matrix}
v_h\\
\psi_h
\end{matrix}\right)}_j = \int_{I_j} w_hv_h dx + \int_{I_{j+\hf}}\chi_h\psi_h dx, \qquad \ip{\cdot,\cdot} = \sum_j\ip{\cdot,\cdot}_j,
\end{equation}
the discrete operator is defined as follows.
\begin{equation}
	\begin{aligned}
	\ip{L_h\left(\begin{matrix}
		w_h\\
		\chi_h
		\end{matrix}\right), \left(\begin{matrix}
		v_h\\
		\psi_h
		\end{matrix}\right)} = &\;\frac{1}{\tau_{\max}} \ip{\left(\begin{matrix}
		\chi_h-w_h\\
		w_h - \chi_h
		\end{matrix}\right), \left(\begin{matrix}
		v_h\\
		\psi_h
		\end{matrix}\right)} + \ip{A\left(\begin{matrix}
		w_h\\
		\chi_h
		\end{matrix}\right), \partial_x\left(\begin{matrix}
		v_h\\
		\psi_h
		\end{matrix}\right)} \\
		&\;+ \sum_j \left(\left(\chi_h[v_h]\right)_{j+\hf} + \left(w_h[\psi_h]\right)_j\right), \qquad \forall 
		\left(\begin{matrix}
		w_h\\
		\chi_h
		\end{matrix}\right),
		\left(\begin{matrix}
		v_h\\
		\psi_h
		\end{matrix}\right)\in V_h.
	\end{aligned}
\end{equation}
For optimal error estimates, Liu et al. designed the projection in \cite[Lemma 2.1]{liu2018optimal} using the shifting technique. We refer to their paper to save space. Then by following arguments in \cite[page 526]{liu2018optimal}, one can verify the following superconvergence result. 
\begin{equation}
	\left|\ip{L_h \Pi \left(
	\begin{matrix}
	w\\
	w
	\end{matrix}\right) - L\left(
	\begin{matrix}
	w\\
	w
	\end{matrix}\right) ,\left(
	\begin{matrix}
	v_h\\
	\psi_h
	\end{matrix}\right)} \right|\leq C h^{k+1} \nm{\left(
	\begin{matrix}
	v_h\\
	\psi_h
	\end{matrix}\right)}.
\end{equation}
 Here $C$ is a constant dependent on the $(k+2)$th Sobolev norm of $w$. The approximation property \eqref{eq-CT} also holds for the constructed projection. Then from Corollary \ref{cor-err}, we have 
\begin{equation}
	\nm{\left(
		\begin{matrix}
		u\\
		u
		\end{matrix}\right) - \left(
		\begin{matrix}
		u_h\\
		\phi_h
		\end{matrix}\right)} \leq C_{\rm{E}} \left(\sigma(\tmu,t^n)+1\right)\left(h^{k+1} + \tau^p\right).
\end{equation}
\end{examp}
\subsection{DG schemes for conservation laws on 2D Cartesian meshes}
\begin{examp}[$Q^k$-DG method]\label{examp-QkDG}
	\upshape
We consider 2D linear scalar conservation laws on rectangular domain $\Omega$ with periodic boundary conditions. 
\begin{equation}\label{eq-2dadv}
	\partial_t u + \left(\partial_{x^1} + \partial_{x^2}\right)u = 0, 	\qquad x = (x^1,x^2) \in \Omega \subset \mathbb{R}^2.
\end{equation}
A quasi-uniform Cartesian mesh is used for discretizing $\Omega = \cup_{j^1,j^2} I_{j^1}\times I_{j^2}$. For the DG discretization, the finite element space is chosen as
\begin{equation}
V_h = V_{h,1}\otimes V_{h,1} = \{v_h\in L^2(\Omega): v_h|_{I_{j^1} \times I_{j^2}} \in Q^k(I_{j^1}\times I_{j^2}),\forall j^1,j^2\}.
\end{equation}
Here $Q^k(I_{j^1}\times I_{j^2}) = P^k(I_{j^1})\otimes P^k(I_{j^2})$.

The DG operator for discretizing $L =  -\left(\partial_{x^1} + \partial_{x^2}\right)$  is defined as follows. 
\begin{equation}\label{eq-DG-2D}
\begin{aligned}
\ip{L_{h} w_h,v_h } = &\; \sum_{j^1,j^2} \int_{I_{j^1}}\int_{I_{j^2}}  w_h \left(\partial_{x^1} +  \partial_{x^2}\right)v_h dx + \sum_{j^1,j^2}\left( \int_{I_{j^2}}\widehat{ w_h}_{j^1+\hf,x^2}^{\theta_1}[v_h]_{j^1+\hf,x^2} dx^2 \right. \\
 &\;+ \left. \int_{I_{j^1}} \widehat{ w_h}_{x^1,j^2+\hf}^{\theta_2}[v_h]_{x^1,j^2+\hf} dx^1\right), \quad \forall w_h,v_h \in V_h.
\end{aligned}
\end{equation}
Here $\theta_1, \theta_2 >\hf$ and 
\begin{eqnarray}
	\widehat{ w_h}_{j^1+\hf,x^2}^{\theta_1} = \theta_1 w_h(x^{1,-}_{j^1+\hf},x^2) +  (1-\theta_1) w_h(x^{1,+}_{j^1+\hf},x^2), \\
	\widehat{ w_h}_{x^1,j^2+\hf}^{\theta_2} = \theta_2 w_h(x^1,x^{2,-}_{j^2+\hf}) +  (1-\theta_2) w_h(x^1,x^{2,+}_{j^2+\hf}).
\end{eqnarray}
It can be shown that $\ip{L_h v_h, v_h} \leq 0$ \cite[Proposition 3.1]{meng2016optimal}. The required projection was first constructed in \cite[Lemma 3.3]{meng2016optimal}, with the regularity assumption improved in \cite[Lemma 3.3]{cheng2017application}.
\begin{PROP}\label{lem-MSW2}
	There exists a uniquely defined $\Pi_{\theta_1,\theta_2}w\in V_h$, such that 
	\begin{eqnarray}
		\int_{I_{j^1}}\int_{I_{j^2}}  \Pi_{\theta_1,\theta_2} w v dx^1dx^2 &=& \int_{I_{j^1}}\int_{I_{j^2}} w v dx^1dx^2, \forall v_h \in Q^{k-1}(I_{j^1}\times I_{j^2}),\\
		\int_{I_{j^2}} (\widehat{ \Pi_{\theta_1,\theta_2}w })_{j^1+\hf,x^2}^{\theta_1}v_h dx^2  &=& 	\int_{I_{j^2}} \widehat{ w }_{j^1+\hf,x^2}^{\theta_1} v_h dx^2, \qquad \forall v \in P^{k-1}(I_{j^2}),\\
		\int_{I_{j^1}} (\widehat{ \Pi_{\theta_1,\theta_2}w })_{x_1,j^2+\hf}^{\theta_2}v_h dx^1  &=&		\int_{I_{j^1}} \widehat{w }_{x_1,j^2+\hf}^{\theta_2}v_h dx^1  , \qquad \forall v \in P^{k-1}(I_{j^1}),\\
		\widehat{\Pi_{\theta_1,\theta_2} w}^{\theta_1,\theta_2}   &=& 	\widehat{w}^{\theta_1,\theta_2}  , \qquad x = (x^1_{j^1+\hf}, x^2_{j^2+\hf}),
		 \qquad \forall j^1,j^2.
	\end{eqnarray}
	Here
	\begin{equation}
	\widehat{w_h}^{\theta_1,\theta_2}  = \theta_1\theta_2w_h^{-,-} + \theta_1(1-\theta_2)w_h^{-,+} + \theta_2(1-\theta_1)w_h^{+,-}  
	+ (1-\theta_1)(1-\theta_2)w_h^{+,+}.
	\end{equation}
	Moreover, 
	$\nm{w - \Pi_{\theta_1,\theta_2} w} \leq Ch^{k+1}
	$,
	where $C$ depend on $(k+1)$th order Sobolev norm of $w$. 
\end{PROP}
The following superconvergence result holds for the projection operator \cite[Lemma 3.6]{cheng2017application}.
\begin{equation}
	|\ip{L_h\Pi w - L w,v_h}|\leq Ch^{k+1} \nm{v_h}.
\end{equation}
Here $C$ depends on $(k+2)$th order Sobolev norm of $w$. 
With the approximation property in Proposition \ref{lem-MSW2}, one can obtain the following error estimate for the fully discrete scheme
\begin{equation}
\nm{u(t^n) - u_h^n }\leq 	C_{\rm{E}}(\sigma(\tmu,t^n)+1)\left( h^{k+1} + \tau^p\right).
\end{equation}

\end{examp}
\begin{examp}[$P^k$-DG method]
	\upshape
	We now consider the DG discretization of \eqref{eq-2dadv} with $P^k$ elements. The settings are similar as that in Example \ref{examp-QkDG}, except for the finite element space replaced with 
	\begin{equation}
		V_h = \{v_h\in L^2(\Omega): v_h|_{I_{j^1} \times I_{j^2}} \in P^k(I_{j^1}\times I_{j^2}),\forall j^1,j^2\}.
	\end{equation}
	Here $P^k(I_{j^1}\times I_{j^2})$ is the space of polynomials of no more than $k$ on $I_{j^1}\times I_{j^2}$. We consider the upwind scheme with $\theta_1 = \theta_2 = 1$. The stability follows closely with that in the previous example. For optimal error estimates, Liu et al. designed the following operator in \cite[Section 2.2.1]{liu2020optimal}.
	\begin{PROP} There exists a uniquely defined $\Pi w \in V_h$, such that
	\begin{align}
		\int_{I_{j^1}}\int_{I_{j^2}} \Pi w dx^1 dx^2 =&\; 		\int_{I_{j^1}}\int_{I_{j^2}} w dx^1 dx^2,\\
		\widetilde{\Pi_h}(\Pi w,v_h)_{j^1,j^2} =&\; 		\widetilde{\Pi_h}(w,v_h)_{j^1,j^2}, \qquad \forall v_h\in P^k(I_{j^1}\times I_{j^2}),\qquad \forall j^1,j^2,
	\end{align}		
	where $\widetilde{\Pi_h}(w,v_h)_{j^1,j^2}$ is defined as follows.
	\begin{equation}
	\begin{aligned}
		\widetilde{\Pi_h}(w,v_h)_{j^1,j^2} =&\; -\int_{I_{j^1}}\int_{I_{j^2}}{w(\partial_{x^1}+\partial_{x^2})v_h}dx^1dx^2 \\
		&\;+ \int_{I_{j^1}}w(x^1,x_{j^2+\hf}^{2,-})\left(v_h(x^1,x_{j^2+\hf}^{2,-})-v_h(x^1,x_{j^2-\hf}^{2,+})\right)dx^1\\
		&\; + \int_{I_{j^1}}w(x_{j^1+\hf}^-,x_{j^2})\left(v_h(x_{j^1+\hf}^-,x^2)-v_h(x_{j^1-\hf}^+,x^2)\right)dx^2.
	\end{aligned}
	\end{equation}
	Furthermore, $\nm{w-\Pi w}\leq Ch^{k+1}$, where $C$ depends on $(k+1)$th order Sobolev norm of $w$.
	\end{PROP}
	Following the lines in \cite[Section 2.2.3]{liu2020optimal}, it can be shown that
	\begin{equation}
		|\ip{L_h\Pi w - L w,v_h}|\leq Ch^{k+1} \nm{v_h}.
	\end{equation}
	Here $C$ depends on $(k+2)$th order Sobolev norm of $w$. With the approximation property above, the following fully discrete error estimate can be obtained. 
	\begin{equation}
	\nm{u(t^n) - u_h^n }\leq 	C_{\rm{E}}(\sigma(\tmu,t^n)+1)\left( h^{k+1} + \tau^p\right).
	\end{equation}	
\end{examp}
\subsection{FG schemes for symmetric hyperbolic systems}\label{sec-fs}
	We consider the linear symmetric hyperbolic system with the periodic boundary condition. 
	\begin{equation}\label{eq-fs}
	\partial_t u +  \sum_{i = 1}^d A_i \partial_{x_i}u = 0,\qquad  x \in \Omega = [0,2\pi]^d \subset \mathbb{R}^d.
	\end{equation}
	Here $\{A_i\}_{i=1}^d$ are $m\times m$ constant symmetric matrices. We have $L = - \sum_{i = 1}^d A_i \partial_{x_i}$ in the example. 
	We then consider FG spatial discretization. To be consistent with notations in the literature,  we switch the subscripts from $h$ to $N$ in this example. This discrete space is 
	\begin{equation}
	V_N = \{v_N\in L^2(\Omega;\mathbb{R}^m): v_N =  \sum_{|k| \leq N} a_k e^{\ii k\cdot x}\},
	\end{equation}
	and the discrete operator is
	$
	L_N  v_N =  - \sum_{i = 1}^d A_i \partial_{x_i} v_N = L v_N$. Semiboundedness of $L_h$ can be verified straightforwardly. 
	For the projection, we take $\Pi = \Pio$. Since $\{A_i\}_{i=1}^d$ are all constant matrices, we have $A_i^\top  v_N \in V_N$. By 
	using the fact that $\Pio \partial_{x_i} w = \partial_{x_i} \Pio w$, it can be shown that
	\begin{equation}
	\begin{aligned}
	\ip{L w, v_N} = &\; \ip{\sum_{i = 1}^dA_i \partial_{x_i}w,v_N} =\sum_{i=1}^d \ip{\partial_{x_i} w, A_i^\top  v_N} =  \sum_{i=1}^d \ip{\Pio \partial_{x_i} w, A_i^\top  v_N}\\
	=&\; \sum_{i=1}^d \ip{\partial_{x_i} \Pio w, A_i^\top  v_N}= \ip{\sum_{i = 1}^d A_i \partial_x\Pio w,v_N} = \ip{L\Pio w,v_N}.
	\end{aligned}
	\end{equation}
	Moreover, we have $\nm{w - \Pio w } \leq CN^{-k}$ if $w\in C_{\rm{p}}^{k}(\Omega;\mathbb{R}^m)$ (subscript $\mathrm{p}$ stands for periodic functions) and $\nm{w - \Pio w } \leq Ce^{-cN}$ if $w$ is analytic, which leads to the following error estimates.
	\begin{eqnarray}
	\nm{u(t^n) - u_N^n}\leq 
	\left\{
	\begin{matrix}
	C_{\rm{E}}(\sigma(\tmu,t^n)+1)\left(N^{-k} + \tau^{p}\right), & u(\cdot,t) \in  C_{\rm{p}}^{k}(\Omega;\mathbb{R}^m),\\
	C_{\rm{E}}(\sigma(\tmu,t^n)+1)\left(e^{-cN} + \tau^{p}\right), & u(\cdot,t) \text{ is analytic}.\\
	\end{matrix}\right.
	\end{eqnarray}

\section{Error estimates: another approach}\label{sec-ano}
\setcounter{equation}{0}
In this section, we consider a round-about argument. We assume there is an error estimate for 
the semidiscrete scheme at hand, and then analyze the error of the RK method by comparing the 
fully discrete solution with the semidiscrete solution, and finally obtain fully discrete error 
with the triangle inequality. This is a feasible approach, however a different type of 
operator $\Pi$ has to be constructed. 

\begin{ASSP}\label{assp-soft}
	\ 
\begin{enumerate}
	\item (Error estimates of method-of-lines scheme). $\nm{u(t) - u_h(t) } \leq E(t,h)$.
	\item There exists $\Pi u(0) \in V_h$, such that $\nm{L_h^{p+1}\Pi u(0)} \leq C_{\Pi}$.	
\end{enumerate}
\end{ASSP}

\begin{LEM}[Error estimates of the time integrator]\label{lem-err}
Under Assumption \ref{assp-rg}, Assumption \ref{assp-Lh}, Assumption \ref{assp-RK}, Assumption \ref{assp-rkacc} and Assumption \ref{assp-soft},  if $u_h^0 = \Pi u(0)$, then 
	\begin{equation}
		\nm{u_h(t^{n})-u_h^{n}} \leq e^{\tmu t^n}\left(\nm{u_h(t^{0})-u_h^{0}} + C_{\star} \tau^p\right),
	\end{equation}
	where $C_{\star} = C_{\Pi}\left(e + {\sum_{i=p+1}^s{|\alpha_i|}}\right)\sigma(\tmu, t^{n})$.
\end{LEM}
\begin{proof}
	We use the convention $\alpha_i = 0$, $\forall i > s$ throughout the proof. Since 
	$
	\sum_{i=0}^\infty\frac{1}{i!}\nm{\tau L_h}^i\leq \sum_{i=0}^\infty \frac{1}{i!} = e<\infty
	$,
	$e^{\tau L_h}=\sum_{i=0}^\infty\frac{1}{i!}(\tau L_h)^i$ is well-defined and 
	\begin{equation}\label{eq-semi-trun}
		u_h(t^{n+1}) = e^{\tau L_h} u_h(t^n) = R_s(\tau L_h) u_h(t^n) + w^n,
	\end{equation}
	where
	\begin{equation}\label{eq-defw}
	\begin{aligned}
		 w^n := &\;\left(e^{\tau L_h}-R_s(\tau L_h)\right)u_h(t^n) 
		 = \tau^{p+1}\left(\sum_{i=p+1}^\infty (\frac{1}{i!}-\alpha_i) (\tau L_h)^{i-(p+1)}\right)L_h^{p+1} u_h(t^n).
	 \end{aligned}
	\end{equation}	
	Subtracting \eqref{eq-fully} from \eqref{eq-semi-trun}, we have
	\begin{equation}
		u_h(t^{n+1})-u_h^{n+1} =  R_s(\tau L_h) \left(u_h(t^n)-u_h^n\right) + w^n.
	\end{equation}
	Therefore,
	\begin{equation}\label{eq-onestep}
	\begin{aligned}
	\nm{u_h(t^{n+1})-u_h^{n+1}} \leq &\;  \nm{R_s(\tau L_h)} \nm{u_h(t^n)-u_h^n} + \nm{w^n}.
	\end{aligned}
	\end{equation}
	By applying the discrete Gr\"onwall's inequality in Proposition \ref{prop-gronwall}, one can obtain
	\begin{equation}
	\begin{aligned}
	&\;\nm{u_h(t^{n+1})-u_h^{n+1}}
	\leq \nm{R_s(\tau L_h)}^{n+1}\nm{u_h(t^{0})-u_h^{0}} + \sum_{k = 0}^{n}\nm{R_s(\tau L_h)}^k \nm{w^{n-k}}.
	\end{aligned}
	\end{equation}
	With the assumption $\nm{R_s(\tau L_h)} \leq 1 + \tmu \tau$, we can use Remark \ref{rmk-dg} to obtain
	\begin{equation}\label{eq-est}
		\nm{u_h(t^{n+1})-u_h^{n+1}}\leq e^{\tmu t^{n+1}}\nm{u_h(t^{0})-u_h^{0}}+ \sigma(\tmu, t^{n+1})\tau^{-1}\max_{0\leq k\leq n} \nm{w^{k}}.
	\end{equation}
	Now we need to estimate $\nm{w^k}$. Use the definition of $w^k$ in \eqref{eq-defw} and it yields that
	\begin{equation}\label{eq-wi}
		\begin{aligned}
		\nm{w^k} 
		 \leq  &\;\left(\sum_{i=p+1}^\infty (\frac{1}{i!}+|\alpha_i|) \nm{\tau L_h}^{i-(p+1)}\right)\nm{L_h^{p+1} u_h(t^k)}\tau^{p+1}\\
		 		 \leq  &\;\left(\sum_{i=p+1}^\infty (\frac{1}{i!}+|\alpha_i|) \lambda^{i-(p+1)}\right)\nm{L_h^{p+1} u_h(t^k)}\tau^{p+1}\\
		\leq &\;\left(e + {\sum_{i=p+1}^s{|\alpha_i|}}\right) \nm{L_h^{p+1} u_h(t^k)} \tau^{p+1}. 
		\end{aligned}
	\end{equation}
	Since 
		$\partial_t L_h^{p+1}u_h(t) = L_h(L_h^{p+1} u_h(t))$, 
	one can apply Proposition \ref{prop-edecay} to obtain
	 \begin{equation}
	 \nm{L_h^{p+1}u_h(t^k)}\leq e^{\mu t^{k}}\nm{L_h^{p+1}u_h(0)} \leq e^{\mu t^{k}}C_{\Pi},\qquad \forall 0\leq k\leq n.
	 \end{equation}
	This together with \eqref{eq-wi} gives
	\begin{equation}\label{eq-wkest}
		\nm{w^k}\leq e^{\mu t^{n+1}} C_{\Pi} \left(e + {\sum_{i=p+1}^s{|\alpha_i|}}\right) \tau^{p+1}, \qquad \forall 0\leq k\leq n .
	\end{equation}
	The proof can be completed by substituting \eqref{eq-wkest} into \eqref{eq-est}. 
\end{proof}

\begin{THM}[Error estimates of the fully discrete scheme]\label{thm-soft}
	Under the same assumptions and the same definition of $C_\star$ as those in Lemma \ref{lem-err}, we have
	\begin{equation}
\nm{{u}(t^{n})-u_h^{n}}\leq e^{\max(\mu,\tmu ) t^n} \left(\nm{ u_{h}(0) - \Pi u(0)} + \nm{\Pi u(0)-u_h^{0}}+ C_{\star} \tau^p\right) + E(t,h).
\end{equation}

\end{THM}
\begin{proof}
	We denote by $\tilde{u}_{h}(t)$ the solution to \eqref{eq-semi} with $\tilde{u}_{h}(0) = \Pi u(0)$ as the initial condition. Using Lemma \ref{lem-err}, one can get
$
		\nm{\tilde{u}_h(t^{n})-u_h^{n}} \leq e^{\tmu t^n}\left(\nm{\Pi u(0)-u_h^{0}} + C_{\star} \tau^p\right).
$
	We then apply the triangle inequality to obtain 
	\begin{equation}\label{eq-u-uhn}
	\begin{aligned}
		\nm{{u}(t^{n})-u_h^{n}}\leq &\; \nm{u(t^n) - {u}_h(t^{n})} + \nm{ u_h(t^n) - \tilde{u}_h(t^{n})}+ \nm{\tilde{u}_h(t^{n})-u_h^{n}}\\
		\leq &\; \nm{u(t^n) - {u}_h(t^{n})} + \nm{ u_h(t^n) - \tilde{u}_h(t^{n})}+  e^{\tmu t^n}\left(\nm{\Pi u(0)-u_h^{0}} + C_{\star} \tau^p\right)\\
		\leq &\;\nm{ u_h(t^n) - \tilde{u}_h(t^{n})} +  E(t,h) + e^{\max(\mu,\tmu )t^n}\left(\nm{\Pi u(0)-u_h^{0}} + C_{\star} \tau^p\right).
	\end{aligned}
	\end{equation}
	By the linearity of the problem, $\tilde{u}_h - u_h$ is also evolved by $
	\frac{\partial}{\partial t} \left( u_{h} - \tilde{u}_{h}\right) = L_h ( u_{h} - \tilde{u}_{h})$. 
	Therefore, due to Proposition \ref{prop-edecay}, we have
	\begin{equation}\label{eq-semi-1}
	\nm{ u_{h}(t^n) - \tilde{u}_{h}(t^n)} \leq e^{\mu t^n} \nm{ u_{h}(0) - \tilde{u}_{h}(0)} \leq  e^{\max(\mu,\tmu ) t^n} \nm{ u_{h}(0) - \Pi u(0)}.
	\end{equation}
	We plug \eqref{eq-semi-1} into \eqref{eq-u-uhn} to complete the proof. 
\end{proof}
	
	The remaining task is to construct $\Pi$ such that $\nm{L_h^{p+1} \Pi u(0)}$ is bounded. Let us consider the advection equation \eqref{eq-adv}, where $L = -\partial_x$. For sufficiently smooth $w$, it has been shown that
	$L_h^{p+1} \Pi w= \Pi_0 L^{p+1} w$ for FG discretization with $\Pi = \Pio$, and for upwind-biased DG discretization with $\Pi w = (D_{h,\theta}^{-1})^{p+1}\Pi_{0} \partial_x^{p+1} w + \frac{\ip{w,1}}{\ip{1,1}}$. Then
	we have $\nm{L_h^{p+1}\Pi u(0)} = \nm{\Pio L^{p+1}u(0)}\leq \nm{L^{p+1}u(0)}$. Therefore, for these schemes, the fully discrete error estimates can also be obtained with Theorem \ref{thm-soft}. For general problems, although we usually have $L_h\Pi w= \Pio Lw$, $L_h^{p+1}\Pi w = \Pio L^{p+1} w$ 
does not hold for $p \geq 1$. Further efforts have to be made for obtaining error estimates 
through these lines.
	
\section{Conclusions}\label{sec-con}
\setcounter{equation}{0}
In this paper, we study the error estimates of fully discrete RKDG schemes for linear time-dependent PDEs. Under the assumptions that the exact solution is sufficiently smooth, the fully discrete scheme is stable and consistent, and there exists a linear spatial operator satisfying certain properties, then we show the fully discrete scheme has the error estimate $\nm{u(t^n)-u_h^n} = 
\mathcal{O}(h^{k+k'}+\tau^p)$ with $k' \in [0,1]$, where $k$ is the DG polynomial degree and
$p$ is the linear order of the time integrator. The error analysis applies to explicit RK methods of any order and to a wide range of DG schemes beyond hyperbolic problems. We have the following highlights for our analysis. First, the required regularity depends on the order of RK methods, but is independent of the number of stages. Second, by concerning a general discrete operator as that in \cite{chen2009unified}, the analysis can be applied to various semidiscrete schemes to time-dependent PDEs. While the current framework has its limitation: it does not have the mechanism to include  ``jump" terms in the DG discretization, hence would only provide suboptimal error estimates for schemes converging at the rate of $k+\hf$. In these cases, a more refined treatment, as that in \cite{xu2019error}, is needed.

\end{document}